\documentclass[12pt, reqno]{amsart} 
\usepackage{amssymb,amscd,amsfonts,amsbsy}
\usepackage{latexsym}
\usepackage{exscale}
\usepackage{amsmath,amsthm,amsfonts}
\usepackage{mathrsfs}
\usepackage{xcolor} 
\usepackage[colorlinks=true,linkcolor=blue,citecolor=red,urlcolor=red, backref=page]{hyperref} 
\usepackage{esint} 

\calclayout

\newtheorem{theorem}{Theorem}
\newtheorem{proposition}[theorem]{Proposition}

\newtheorem{lemma}[theorem]{Lemma}

\theoremstyle{definition}
\newtheorem{definition}{Definition}

\newtheorem{remark}[definition]{Remark}

\numberwithin{theorem}{section}
\numberwithin{definition}{section}
\numberwithin{equation}{section}

\def\dist{\operatorname{dist}\,}

\def\C{\mathbb{C}}
\def\D{\mathcal{D}}
\def\R{\mathbb{R}}
\def\Rn{{\mathbb{R}^n}}

\def\N{\mathbb{N}}
\def\Q{{\mathcal{Q}}}
\def\Z{\mathbb{Z}}

\def\G{\mathscr{G}}
\def\ch{\operatorname{ch}}
\def\supp{\operatorname{supp}}
\def\dist{\operatorname{dist}}

\begin{document}

\author{Mingming Cao}
\address{
Instituto de Ciencias Matem\'aticas CSIC-UAM-UC3M-UCM\\
Con\-se\-jo Superior de Investigaciones Cient{\'\i}ficas\\
C/ Nicol\'as Cabrera, 13-15\\
E-28049 Ma\-drid, Spain} 
\email{mingming.cao@icmat.es}

\author{Qingying Xue}
\address{
         School of Mathematical Sciences \\
         Beijing Normal University \\
         Laboratory of Mathematics and Complex Systems \\
         Ministry of Education \\
         Beijing 100875 \\
         People's Republic of China}
\email{qyxue@bnu.edu.cn}

\thanks{The first author acknowledges financial support from the Spanish Ministry of Science and Innovation, through the ``Severo Ochoa Programme for Centres of Excellence in R\&D'' (SEV-2015-0554) and from the Spanish National Research Council, through the ``Ayuda extraordinaria a Centros de Excelencia Severo Ochoa'' (20205CEX001). The second author was supported partly by NSFC {(Nos. 11671039, 11871101) and NSFC-DFG (No. 11761131002).\\
Corresponding author: Qingying Xue. Email: qyxue@bnu.edu.cn}}


\date{July 26, 2020.}
\title[Multilinear Littlewood-Paley-Stein operators]{Multilinear Littlewood-Paley-Stein Operators on Non-homogeneous Spaces}
\maketitle

\begin{abstract}Let $\kappa \ge 2, \lambda > 1$ and define the multilinear Littlewood-Paley-Stein operators by
\begin{align*}
g_{\lambda,\mu}^*(\vec{f})(x)
= \bigg(\iint_{\R^{n+1}_{+}} \vartheta_t(x, y)
\bigg|\int_{\R^{n \kappa}} s_t(y,\vec{z}) \prod_{i=1}^{\kappa} f_i(z_i) \ d\mu(z_i)\bigg|^2 
\frac{d\mu(y) dt}{t^{m+1}}\bigg)^{\frac12}, 
\end{align*} 
where $\vartheta_t(x, y)=\big(\frac{t}{t + |x - y|}\big)^{m \lambda}$. 
In this paper, our main aim is to investigate the boundedness of $g_{\lambda,\mu}^*$ on non-homogeneous spaces. By means of probabilistic and dyadic techniques, together with non-homogeneous analysis, we show that $g_{\lambda,\mu}^*$ is bounded from $L^{p_1}(\mu) \times \cdots \times L^{p_{\kappa}}(\mu)$ to $L^p(\mu)$ under certain weak type assumptions. The multilinear non-convolution type kernels $s_t$ only need to satisfy some weaker conditions than the standard conditions of multilinear Calder\'{o}n-Zygmund type kernels and the measures $\mu$ are only assumed to be upper doubling measures (non-doubling). The above results are new even under Lebesgue measures. This was done by considering first a sufficient condition for the strong type boundedness of $g_{\lambda,\mu}^*$ based on an endpoint assumption, and then directly deduce the strong
bound on a big piece from the weak type assumptions.
\end{abstract}


\section{Introduction}
Littlewood-Paley type operators, including $g$-function, Lusin area integral $S$, $g_{\lambda}^{*}$-function and Marcinkiewicz integral $\mu$, have been the subjects of many recent
research works in Harmonic analysis, function spaces and PDE. The first theorem about Littlewood-Paley operators was given by Littlewood and Paley \cite{LP-1} for $g$-function in their study of the Fourier series. Subsequently, they \cite{LP-2} introduced $g_{\lambda}^{*}$-function and established its $L^p$ bondedness. The above results were
extended to Lusin area integral $S$ and Marcinkiewicz integral by Marcinkiewicz and Zygmund \cite{MZ,Z}. It is worth pointing out that the results obtained for $\mu$ depended heavily on complex
function theory, and thus they were severely limited to the case of one dimension.
\medskip

It was Stein \cite{S-58, S-61} who extended the definitions and the results of the above operators to higher dimensions. The weak type estimates and $L^p$ boundedness of these
operators were obtained by the method of real variables, which opened the door to investigate Littlewood-Paley type operators in a wide variety of spaces, such as Hardy spaces
\cite{FS}, $BMO$ spaces \cite{K}, weighted spaces \cite{GW,MR}, Sobolov spaces \cite{SW}, and Campanato spaces \cite{SY}. Furthermore, the operators studied by Stein and their variations
could be used as basic tools in the study of PDE, see for example
\cite{CDM, CMM, DJK, DJ}. Moreover, many other publications about the improvements and variants of Littlewood-Paley-Stein operators came to enrich the literature on Littlewood-Paley
theory (cf. e.g. \cite{K-1, L, L-1, L-2, L-3, RS}).
To state more conveniently, recall that the classical $g_{\lambda}^*$-function of higher dimension $(n \geq 2)$ defined by Stein are as follows:
$$
g_{\lambda}^*(f)(x)
=\bigg(\iint_{\R^{n+1}_{+}} \Big(\frac{t}{t+|x-y|}\Big)^{n\lambda}
|\nabla P_t * f(y)|^2 \frac{dy dt}{t^{n-1}}\bigg)^{1/2},\quad\quad \lambda > 1
$$
where $P_t(y)=t^{-n} P(t^{-1} y)$, $P$ is the Poisson kernel and
$\nabla =(\frac{\partial}{\partial y_1},\ldots,\frac{\partial}{\partial y_n},\frac{\partial}{\partial t})$.
It was shown by Stein that $g_{\lambda}^*$ is of weak type $(1,1)$ for $\lambda > 2$, and is of strong type $(p,p)$ for $1 < p < \infty$. Stein also pointed out that the weak type $(1,1)$
estimate doesn't hold for $1<\lambda\le 2$. Later on, as a replacement of weak $(1,1)$ bounds for $1<\lambda<2$, Fefferman \cite{F} established the endpoint weak $(p,p)$ estimate of
$g_{\lambda}^*$-function when $p>1$ and $\lambda=2/p$. Obviously, for any $x \in \Rn$, the smaller $\lambda$ the greater $g_{\lambda}^*(f)(x)$. This basic fact implies that the properties of
$g_{\lambda}^*$ depend critically on the appropriate relationship between $p$ and $\lambda$. This makes the study of $g_{\lambda}^*$ pretty much interesting, but also, makes its study
more difficult than $g$-function and Lusin area integral $S$.
\medskip

The purpose of this paper is to study the multilinear Littlewood-Paley-Stein $g_{\lambda}^*$-function and Lusin area integral $S$ on non-homogeneous spaces.
We only focus on discussing $g_{\lambda}^*$-function by the reason that Lusin area integral $S$ is pointwisely controlled by $g_{\lambda}^*$-function. The methods we will use are
beyond doubling measures and classical methods, which are replaced by probabilistic methods, multilinear dyadic martingale and non-homogeneous analysis. We will describe these
components in a more detail way later.

Our object of investigation is the multilinear Littlewood-Paley-Stein $g_{\lambda}^*$-function as follows:
\begin{align*}
g_{\lambda,\mu}^*(\vec{f})(x)
= \bigg(\iint_{\R^{n+1}_{+}} \Big(\frac{t}{t + |x - y|}\Big)^{m \lambda} |\Theta_t^\mu (\vec{f})(y)|^2
\frac{d\mu(y) dt}{t^{m+1}}\bigg)^{1/2},\ \ \lambda > 1,
\end{align*}
where $\mu$ is a non doubling measure and the multilinear form $\Theta_t^\mu$ is defined by
$$
\Theta_t^\mu (\vec{f})(y)
= \int_{(\Rn)^{\kappa}} s_t(y,\vec{z}) \prod_{i=1}^{\kappa} f_i(z_i) \ d\mu(z_1) \cdots d\mu(z_{\kappa}).
$$

Obviously, the classical $g_{\lambda}^*$ function was extended to the multilinear setting. Moreover, it also extends Shi, Xue and Yabuta's \cite{SXY} multilinear operator with convolution type kernels and Lebesgue
measure to non-convolution type kernels and non-doubling measures. It was well-known that the theory of multilinear Littlewood-Paley-Stein operators originated in the works of Coifman and Meyer
\cite{CM}. Soon after, the importance of this kind of multilinear operators was shown in PDE and other fields. In 1982, Fabes, Jerison and Kenig \cite{FJK-1} first obtained some multilinear
Littlewood-Paley-Stein type estimates for the square root of an elliptic operator in divergence form, and then studied the solutions of Cauchy problem for non-divergence form parabolic equations. In
1984, based on a multilinear Littlewood-Paley-Stein estimate, the authors \cite{FJK-2} gave necessary and sufficient conditions for absolute continuity of elliptic-harmonic measure.
Additionally, in 1985, Fabes et al. \cite{FJK-3} investigated a class of multilinear square functions and applied it to Kato's problem. Beyond all these, in terms of the weighted
theory, the latest developments about the multilinear Littlewood-Paley-Stein operators can be found in
\cite{BH, CXY, HXMY, XY}.
Recently, probabilistic methods and dyadic analysis has attracted renewed attention because of the celebrated $A_2$ conjecture \cite{Ht}.  
All his work were based on an improvement of random dyadic grids and probabilistic methods, which were introduced by Nazarov, Treil and Volberg [34] in the study of Calder\'{o}n-Zygmund operators on non-homogeneous spaces. Hyt\" {o}nen's work has inaugurated a new research direction in probability theory and Harmonic analysis. Later on, there is a large literature adopting the ideology of Hyt\" {o}nen both one-parameter and multi-parameter cases, homogeneous and non-homogeneous setting. 
These powerful tools have not widely used in multilinear Harmonic analysis. In this regard the current paper is a continuation of the
recent development in the probabilistic methods. Additionally, it extends the results in \cite{CX-3} to the multilinear setting.
\medskip

This paper is organized as follows. In Section \ref{Main results}, we formulate the main results of this paper. Some standard and general estimates will be given in Section \ref{Sec-estimates}, which will contribute to the endpoint bound of $g_{\lambda}^*$-function and to establish the good lambda type inequality. Then, in Section \ref{Sec-main}, we will complete the proof of multilinear local $T1$ theorem. Section \ref{Sec-lambda} will be devoted to demonstrate the strong type bound $L^{p_1}(\mu) \times \cdots \times L^{p_{\kappa}}(\mu) \rightarrow L^{p}(\mu)$ of $g_{\lambda}^*$-function based on an endpoint priori assumption. Finally, in Section \ref{Sec-big}, we demonstrate a big piece multilinear local $T1$ theorem. One should also noted that our more general non-convolution
type kernel conditions are weaker than the convolution-type conditions in \cite{SXY} and the measures now we will work with are non-doubling measures.

\vspace{0.3cm}

\section{Main results}\label{Main results}
Let $\mathfrak{M}(\Rn)$ be the space of all complex Borel measures in $\Rn$ equipped with the norm of total variation $||\nu|| = |\nu|(\Rn)$.
Recall that, a Borel measure $\mu$ on $\Rn$ is called $a \ power \ bounded$ $measure$, if for some $m>0$, it satisfies
$$
\mu(B(x,r)) \lesssim r^m, \ \ x \in \Rn, \ r>0.
$$
\begin{definition}\label{Def}
Let $\kappa$ be a positive integer and $\mu$ be a power bounded measure. Given a vector of complex measures $\vec{\nu}=(\nu_1,\ldots,\nu_{\kappa})$, we define the multilinear
Littlewood-Paley-Stein $g_\lambda^*$-function as follows
\begin{align*}
g_{\lambda}^*(\vec{\nu})(x)
= \bigg(\iint_{\R^{n+1}_{+}} \Big(\frac{t}{t + |x - y|}\Big)^{m \lambda} |\Theta_t (\vec{\nu})(y)|^2
\frac{d\mu(y) dt}{t^{m+1}}\bigg)^{1/2},\ x \in \Rn,\ \lambda > 1,
\end{align*}
where
$$
\Theta_t (\vec{\nu})(y) = \int_{(\Rn)^{\kappa}} s_t(y,\vec{z})\ d\nu_1(z_1) \cdots d\nu_{\kappa}(z_{\kappa}).
$$
The kernel $s_t : (\Rn)^{\kappa+1} \rightarrow \C$ is assumed to satisfy the following estimates: for some
$\alpha>0$
\begin{enumerate}
\item [(1)] Size condition :
$$
|s_t(x,\vec{y})| \lesssim \frac{t^{\kappa \alpha}}{\prod_{i=1}^{\kappa}(t + |x - y_i|)^{m+\alpha}}.
$$
\item [(2)] H\"{o}lder conditions :
$$
|s_t(x,\vec{y}) - s_t(x',\vec{y})| \lesssim \frac{t^{(\kappa-1)\alpha} |x - x'|^{\alpha}}{\prod_{i=1}^{\kappa}(t + |x - y_i|)^{m+\alpha}},
$$
whenever $|x - x'| < t/2$ and
$$
|s_t(x,\vec{y}) - s_t(x,y_1,\ldots,y_i',\ldots,y_{\kappa})| \lesssim \frac{t^{(\kappa-1)\alpha} |y_i - y_i'|^{\alpha}}{\prod_{i=1}^{\kappa}(t + |x - y_i|)^{m+\alpha}},
$$
whenever $|y_i - y_i'| < t/2$ for all $1 \leq i \leq \kappa$.
\end{enumerate}
\end{definition}
In particular, we denote
\begin{align*}
g_{\lambda,\mu}^*(\vec{f})(x)
= \bigg(\iint_{\R^{n+1}_{+}} \Big(\frac{t}{t + |x - y|}\Big)^{m \lambda} |\Theta_t^\mu (\vec{f})(y)|^2
\frac{d\mu(y) dt}{t^{m+1}}\bigg)^{1/2},\ x \in \Rn,\ \lambda > 1,
\end{align*}
where
$$
\Theta_t^\mu (\vec{f})(y) 
= \int_{(\Rn)^{\kappa}} s_t(y,\vec{z}) \prod_{i=1}^{\kappa} f_i(z_i) \ d\mu(z_1) \cdots d\mu(z_{\kappa}).
$$

We also need the local version of $g_{\lambda}^*$ and $g_{\lambda,\mu}^*$.
For a given cube $Q$, the local $g_{\lambda}^*$-function is defined by
\begin{align*}
g_{\lambda,Q}^*(\vec{\nu})(x)
= \bigg(\int_{0}^{\ell(Q)}\int_{\Rn} \Big(\frac{t}{t + |x - y|}\Big)^{m \lambda} |\Theta_t (\vec{\nu})(y)|^2
\frac{d\mu(y) dt}{t^{m+1}}\bigg)^{1/2},\ \lambda > 1.
\end{align*}
Similarly, the local $g_{\lambda,\mu}^*$-function is defined in the way that $g_{\lambda,\mu,Q}^*(\vec{f})=g_{\lambda,Q}^*(f_1 \mu,\ldots,f_{\kappa} \mu)$.

Now, we give the definition of $(a,b)$-doubling measure condition and the $\mathfrak{C}$-small boundary condition.
\begin{definition}
\begin{enumerate}
\item [(1)] Given $a,b>1$, a cube $Q \subset \Rn$ is called $(a,b)$-doubling for a given measure $\mu$ if $\mu(a Q) \leq b \mu(Q)$.
\item [(2)] Given $\mathfrak{C}>0$ we say that a cube $Q \subset \Rn$ has $\mathfrak{C}$-small boundary with respect to the measure $\mu$ if
$$
\mu \big(\{x \in 2Q; \dist(x,\partial Q) \leq \xi \ell(Q) \}\big) \leq \mathfrak{C} \xi \mu(2Q)
$$
for every $\xi > 0$.
\end{enumerate}
\end{definition}
The main result of this paper is the following.  
\begin{theorem}\label{Local T1}
Let $\lambda > 2 \kappa$, $0 < \alpha \leq m(\lambda-2\kappa)$ and $1 < p_1, \cdots,p_{\kappa}<\infty$ with $\frac{1}{p}=\frac{1}{p_1} + \cdots + \frac{1}{p_\kappa}$. Assume that
$\mu$ is a power bounded measure, $p_0>0$, $\delta_0 < 1$ and $C_0 < \infty$ are given constants. Let
$\beta > 0$ and $\mathfrak{C}$ be large enough depending only on $n$. Suppose that for every $(2, \beta)$-doubling cube $Q \subset \Rn$ with $\mathfrak{C}$-small boundary, there
exists
$H_Q \subset \Rn$ such that $\mu(H_Q) \leq \delta_0 \mu(Q)$ and
$$
\sup_{\zeta >0} \zeta^{p_0} \mu \big(\{x \in Q \setminus H_Q; g_{\lambda,\mu,Q}^*(\mathbf{1}_Q,\cdots,\mathbf{1}_Q) > \zeta\}\big) \leq C_0 \mu(Q).
$$
Then we have
$$
\big\| g_{\lambda,\mu}^*(\vec{f}) \big\|_{L^p(\mu)}
\lesssim \prod_{i=1}^{\kappa} \big\|f_i\big\|_{L^{p_i}(\mu)}.
$$
\end{theorem}

To show the above main theorem, we need to give a sufficient condition for the strong type boundedness based on an endpoint assumption.
\begin{theorem}\label{L^p}
Let $\lambda > 2 \kappa$, $0 < \alpha \leq m(\lambda-2\kappa)$ and $1 < p_1, \cdots,p_{\kappa}<\infty$ with $\frac{1}{p}=\frac{1}{p_1} + \cdots + \frac{1}{p_\kappa}$.
Assume that $\mu$ is a power bounded measure. Let $\beta > 0$ and $\mathfrak{C}$ be the big enough numbers, depending only on the dimension $n$, and $\theta \in (0,1)$. Suppose that
for each $(2,\beta)$-doubling cube $Q$ with $\mathfrak{C}$-small boundary, there exists a subset $G_Q \subset Q$ such that $\mu(G_Q) \geq \theta \mu(Q)$ and
$g_{\lambda}^* : \mathfrak{M}(\Rn) \times \cdots \times \mathfrak{M}(\Rn) \rightarrow L^{\frac{1}{\kappa},\infty}(\mu \lfloor G_Q)$ is bounded with a uniform constant independent of
$Q$.
Then there holds that
$$
\big\| g_{\lambda,\mu}^*(\vec{f}) \big\|_{L^p(\mu)}
\lesssim \prod_{i=1}^{\kappa} \big\|f_i\big\|_{L^{p_i}(\mu)}.
$$
\end{theorem}

Moreover, we may directly deduce the strong bound on a big piece from the weak type assumption in Theorem \ref{Local T1}. We will see that it needs some delicate arguments to obtain Theorem \ref{Local T1} from the result below.  
\begin{theorem}\label{Big piece}
Let $\lambda > 2 \kappa$, $0 < \alpha \leq m(\lambda-2\kappa)$ and $1 < p,p_1, \cdots,p_{\kappa}<\infty$ with $\frac{1}{p}=\frac{1}{p_1} + \cdots + \frac{1}{p_\kappa}$.
Suppose that $\mu$ is a power bounded measure, $Q \subset \Rn$ is a fixed cube. Assume that for some $p_0>0$ and for some $H_Q \subset \Rn$ satisfying $\mu(H_Q) \leq \delta_0 \mu(Q)$, there holds that
\begin{equation}\label{Weak}
\sup_{\zeta > 0} \zeta^{p_0} \mu \big(\{x \in Q \setminus H_Q; g_{\lambda,\mu,Q}^*(\mathbf{1}_Q,\cdots,\mathbf{1}_Q)(x) > \zeta \}\big)
\leq C_0 \mu(Q).
\end{equation}
Then there exists $G_Q \subset Q \setminus H_Q$ so that $\mu(G_Q) \geq \frac{1-\delta_0}{2} \mu(Q)$ and
$$
\big\| \mathbf{1}_{G_Q} g_{\lambda,\mu}^*(\vec{f}) \big\|_{L^p(\mu)}
\lesssim \prod_{i=1}^{\kappa}\big\|f_i\big\|_{L^{p_i}(\mu)}
$$
for each $f_i \in L^{p_i}(\mu)$ with $\supp(f_i) \subset Q$, $i=1,\cdots,\kappa$.
\end{theorem}

For simplicity, we only give the proofs for the case $\kappa=2$. And the general case can be demonstrated similarly but with more complicated calculations and symbols.

\begin{remark}
The above theorems can be extended to more general non-doubling measures.
\begin{enumerate}
\item
Let $\lambda : \Rn \times (0,\infty) \rightarrow (0,\infty)$ be a function so that $r \mapsto \lambda(x,r)$ is non-decreasing for all $x \in \Rn$ and $r>0$. We say that a Borel
measure $\mu$ in $\Rn$ is upper doubling \cite{Ht2} with the dominating function $\lambda$, if there holds that
$$
\mu(B(x,r)) \leq \lambda(x,r) \leq C_{\lambda} \lambda(x,r/2),\ \ x \in \Rn, \ r > 0.
$$
Then we define the $g_{\lambda}^*$-function adapted to the upper doubling measure $\mu$ :
\begin{align*}
g_{\epsilon,\mu}^*(\vec{\nu})(x)
= \bigg(\iint_{\R^{n+1}_{+}} \vartheta_{t,\epsilon}(x,y) |\Theta_t (\vec{\nu})(y)|^2
\frac{d\mu(y)}{\lambda(x,t)} \frac{dt}{t}\bigg)^{1/2},
\end{align*}
where $\Theta_t (\vec{\nu})$ is the same as that in Definition \ref{Def} and
$$
\vartheta_{t,\epsilon}(x,y)
:=\frac{t^{\epsilon_1} \lambda(x,t)^{\epsilon_2}}{t^{\epsilon_1} \lambda(x,t)^{\epsilon_2} + |x-y|^{\epsilon_1} \lambda(x,|x-y|)^{\epsilon_2}}, \quad \epsilon_1>0, \epsilon_2>2m+1.
$$
\item
The multilinear Lusin area integral $S$ associated with the upper doubling measure $\mu$ is defined by
$$
S_{\mu}(\vec{f})(x):=\bigg(\int_{0}^{\infty} \int_{\Gamma(x,t)} |\Theta_t^\mu (\vec{f})(y)|^2 \frac{d\mu(y)}{\lambda(x,t)} \frac{dt}{t}\bigg)^{1/2},
$$
where $\Gamma(x,t)=\{y \in \Rn;|x-y| \leq t\}$.
\end{enumerate}
Theorems $\ref{Local T1}$, $\ref{L^p}$ and $\ref{Big piece}$ also hold for $g_{\epsilon,\mu}^*$ and $S_{\mu}$ with the upper doubling measure $\mu$.
\end{remark}

\section{Some Standard Estimates}\label{Sec-estimates}
The goal of this section is to establish several important key lemmas, which will be applied in the endpoint estimate and to establish a good lambda inequality.
\begin{lemma}\label{U(f)}
For any $x,x_0 \in \Rn$ and $t>0$, we have the pointwise domination :
\begin{equation}\label{U-L}
\mathscr{U}_t(\vec{f})(x):= \bigg(\int_{\Rn} \Big(\frac{t}{t + |x - y|}\Big)^{m \lambda} |\Theta^{\mu}_t (\vec{f})(y)|^2 \frac{d\mu(y)}{t^m}\bigg)^{1/2}
\lesssim \prod_{i=1}^{2} \mathscr{L}_t(f_i)(x),
\end{equation}
and
\begin{equation}\label{U-U-L}
\big| \mathscr{U}_t(\vec{f})(x) - \mathscr{U}_t(\vec{f})(x_0) \big|
\lesssim t^{-1} |x-x_0| \prod_{i=1}^{2} \mathscr{L}_t(f_i)(\bar{x}),
\end{equation}
where $\bar{x}=x_0+\theta (x-x_0)$ and
$$
\mathscr{L}_t(f)(x):= \int_{\Rn}\frac{t^{\alpha/4}}{(t+|x-z|)^{m+\alpha/4}}|f(z)|d\mu(z).
$$
\end{lemma}
\begin{proof}
The inequality $(\ref{U-U-L})$ is a simple application of $(\ref{U-L})$. Actually,
\begin{align*}
\mathscr{P}_{t}(y)
&:=\bigg| \bigg(\frac{t}{t + |x - y|}\bigg)^{m \lambda/2} - \bigg(\frac{t}{t + |x_0 - y|}\bigg)^{m \lambda/2} \bigg|
\\
&\lesssim \frac{|x-x_0|}{t} \bigg(\frac{t}{t + |\bar{x} - y|}\bigg)^{m \lambda/2},
\end{align*}
where $\bar{x}=x_0+\theta (x-x_0)$.
This implies that
\begin{align*}
\big| \mathscr{U}_t(\vec{f})(x) - \mathscr{U}_t(\vec{f})(x_0) \big|
&\leq \bigg(\int_{\Rn} \mathscr{P}_{t}(y)^2 |\Theta^{\mu}_t (\vec{f})(y)|^2 \frac{d\mu(y)}{t^m}\bigg)^{1/2} \\
&\lesssim t^{-1} |x-x_0| \mathscr{U}_t(\vec{f})(\bar{x})
\lesssim t^{-1} |x-x_0| \prod_{i=1}^{2} \mathscr{L}_t(f_i)(\bar{x}).
\end{align*}

In order to obtain $(\ref{U-L})$, we split the underlying space into four pieces :
\begin{eqnarray*}
\Xi_1 &:=& \big\{y \in \Rn; |y-z_i| \leq |x-z_i|/2,\ i=1,2 \big\}, \\
\Xi_2 &:=& \big\{y \in \Rn; |y-z_i| > |x-z_i|/2,\ i=1,2 \big\}, \\
\Xi_3 &:=& \big\{y \in \Rn; |y-z_1| \leq |x-z_1|/2, |y-z_2| > |x-z_2|/2 \big\}, \\
\Xi_4 &:=& \big\{y \in \Rn; |y-z_1| > |x-z_1|/2, |y-z_2| \leq |x-z_2|/2 \big\}.
\end{eqnarray*}
In the first case, there holds that
$$
|x-y| \geq |x-z_i| - |y-z_i| \geq |x-z_i|/2,\ i=1,2.
$$
Note that
$$
\Big(\frac{t}{t + |x - y|}\Big)^{m \lambda}
\leq \Big(\frac{t}{t + |x - y|}\Big)^{4m + \alpha}
\lesssim  t^{4m} \prod_{i=1}^2 \frac{t^{\alpha/2}}{(t + |x - z_i|)^{2m + \alpha/2}}.
$$
Hence, it yields that
\begin{align*}
\mathscr{U}_{t,1}(\vec{f})(x)
&\lesssim \int_{\R^{2n}} \bigg(\int_{\Rn} \prod_{i=1}^2 \frac{t^{2m + 2 \alpha}}{(t + |y - z_i|)^{2m + 2 \alpha}} \frac{d\mu(y)}{t^m} \bigg)^{1/2}
\\
&\qquad\times \prod_{i=1}^2 \frac{t^{\alpha/4}}{(t + |x - z_i|)^{m + \alpha/4}} |f_i(z_i)| d\mu(z_i) \\
&\lesssim \mathscr{L}_t(f_1)(x) \mathscr{L}_t(f_2)(x).
\end{align*}
It is easy to handle the second term.
\begin{align*}
\mathscr{U}_{t,2}(\vec{f})(x)
&\lesssim \bigg(\int_{\Rn} \Big(\frac{t}{t + |x - y|}\Big)^{m \lambda} \frac{d\mu(y)}{t^m} \bigg)^{1/2}
\\
&\qquad\times \prod_{i=1}^2 \int_{\Rn} \frac{t^{\alpha}}{(t + |x - z_i|)^{m + \alpha}} |f_i(z_i)| d\mu(z_i) \\
&\lesssim \mathscr{L}_t(f_1)(x) \mathscr{L}_t(f_2)(x).
\end{align*}
As for the third term, we notice the facts that
$$
\Big(\frac{t}{t + |x - y|}\Big)^{m \lambda/2}
\leq \Big(\frac{t}{t + |x - y|}\Big)^{2m + \alpha/2}
\lesssim  t^{2m} \frac{t^{\alpha/2}}{(t + |x - z_1|)^{2m + \alpha/2}},
$$
and
$$
\frac{t^{2 \alpha}}{(t + |y - z_2|)^{2m + 2\alpha}}
\lesssim \frac{t^{2 \alpha}}{(t + |x - z_2|)^{2m + 2\alpha}}
\leq \frac{t^{\alpha}}{(t + |x - z_2|)^{2m + \alpha}} .
$$
Then we deduce that
\begin{align*}
\mathscr{U}_{t,3}(\vec{f})(x)
&\lesssim \int_{\R^{2n}} \bigg(\int_{\Rn} \frac{t^{2m + 2 \alpha}}{(t + |y - z_1|)^{2m + 2 \alpha}} \frac{d\mu(y)}{t^m} \bigg)^{1/2} \\
&\qquad\times \prod_{i=1}^2 \frac{t^{\alpha/4}}{(t + |x - z_i|)^{m + \alpha/4}} |f_i(z_i)| d\mu(z_i) \\
&\lesssim \mathscr{L}_t(f_1)(x) \mathscr{L}_t(f_2)(x).
\end{align*}
The last term is symmetric with the third one. This completes the proof.

\end{proof}
\begin{lemma}\label{Omega}
Let $f_i$ (i=1,2) be a bounded function and has a compact support. For every $t_0 > 0$, the $t_0$-truncated version of $g_{\lambda,\mu}^*(f)$ is defined by
$$
g_{\lambda,\mu,t_0}^*(f_1,f_2)(x)
= \bigg(\int_{t_0}^{\infty} \int_{\Rn} \Big(\frac{t}{t + |x - y|}\Big)^{m \lambda} |\Theta_t^\mu (f_1,f_2)(y)|^2
\frac{d\mu(y) dt}{t^{m+1}}\bigg)^{\frac12}.
$$
Set
$$
\Omega_{\xi}:=\big\{x\in \Rn; g_{\lambda,\mu,t_0}^{*} (f_1,f_2)(x) > \xi \big\}, \ \ \text{for any} \ \xi >0.
$$
Then $\Omega_{\xi} \neq \Rn$, $\mu(\Omega_{\xi})<\infty$ and $\Omega_{\xi}$ is an open set.
\end{lemma}
\begin{proof}
We begin by showing that $\Omega_{\xi} \neq \Rn$ and $\mu(\Omega_{\xi})<\infty$. Let $r > 0$ such that $\supp f_i \subset B(0,r)$.
From Lemma $\ref{U(f)}$, for $t \geq t_0$, it follows that
\begin{align*}
\mathscr{U}_t(\vec{f})(x)
&\lesssim \prod_{i=1}^2 ||f_i||_{L^{\infty}(\mu)} \int_{B(0,r)} \frac{ d\mu(z_i)}{(t+|x-z_i|)^{m}} \\
&\lesssim C_{\vec{f}} \ \frac{r^{2m}}{(t+\dist(x,B(0,r)))^{2m}} \\
&\leq C_{\vec{f}} \ \frac{r^{2m}}{(t_0+\dist(x,B(0,r)))^{2m-\epsilon}} \frac{1}{t^{\epsilon}},
\end{align*}
where $\varepsilon \in (0,m(1-1/p))$. Then it yields that
\begin{equation}\label{C(f)}
g_{\lambda,\mu,t_0}^*(f)(x)
\leq C_{\vec{f},t_0} \frac{r^{2m}}{(t_0+\dist(x,B(0,r)))^{2m-\epsilon}},
\end{equation}
which gives that
\begin{equation}\label{Prior}
\big\| g_{\lambda,\mu,t_0}^*(\vec{f}) \big\|_{L^p(\mu)}
\leq C_{\vec{f},t_0} r^{2m} \bigg(\int_{\Rn} \frac{d\mu(x)}{(t_0+\dist(x,B(0,r)))^{p(2m-\epsilon)}} \bigg)^{1/p}
< \infty.
\end{equation}
Moreover, the inequality $(\ref{C(f)})$ also indicates that $$\lim\limits_{|x| \rightarrow \infty}g_{\lambda,\mu,t_0}^*(f)(x) = 0.$$
Thus, there exists a constant $R_0>0$ such that $\Omega_{\xi} \subset B(0,R_0)$, which implies that $\Omega_{\xi} \neq \Rn$ and $\mu(\Omega_{\xi}) < \infty$.

Then, in order to show $\Omega_{\xi}$ is an open set, it suffices to demonstrate the map $x \mapsto g_{\lambda,\mu,t_0}^*(f)(x)$ is continuous.
It is easy to see that
$$
\big| g_{\lambda,\mu,t_0}^*(\vec{f})(x) - g_{\lambda,\mu,t_0}^*(\vec{f})(x_0) \big|
\leq  \bigg(\int_{t_0}^{\infty} \big| \mathscr{U}_t(\vec{f})(x) - \mathscr{U}_t(\vec{f})(x_0) \big|^2 \frac{dt}{t}\bigg)^{1/2}.
$$
For any $t \geq t_0$, it follows from ($\ref{U-U-L}$) that 
\begin{align*}
\big| \mathscr{U}_t(\vec{f})(x) - \mathscr{U}_t(\vec{f})(x_0) \big|
&\lesssim \frac{|x-x_0|}{t^{1-\alpha/2}} \prod_{i=1}^2 \int_{\Rn} \frac{|f_i(z_i)|}{(t+|\bar{x}-z_i|)^{m+\alpha/4}} d\mu(z_i) \\
& \lesssim \frac{|x-x_0|}{t^{1-\alpha_0}} \prod_{i=1}^2 ||f_i||_{L^{p_i}(\mu)} \bigg(\int_{\Rn} \frac{d\mu(z_i)}{(t_0+|\bar{x}-z_i|)^{(m+\alpha_0 )p_i'}} \bigg)^{1/{p_i'}} \\
& \leq C_{t_0} \frac{|x-x_0|}{t^{1-\alpha_0}} \prod_{i=1}^2 ||f_i||_{L^{p_i}(\mu)},
\end{align*}
where the auxiliary number $\alpha_0 \in (0,1)$. Therefore, we deduce that
$$
\big| g_{\lambda,\mu,t_0}^*(f)(x) - g_{\lambda,\mu,t_0}^*(f)(x_0) \big|
\leq C_{t_0} |x-x_0| \prod_{i=1}^2 ||f_i||_{L^{p_i}(\mu)},
$$
which implies the continuity of $x \mapsto g_{\lambda,\mu,t_0}^*(f)(x)$.
This proves Lemma $\ref{Omega}$.
\end{proof}

\begin{lemma}\label{T(f)}
Let $c_0$ be a positive constant, $Q$ be a cube and $x,x' \in Q$. Let $f_i^0 = f_i \mathbf{1}_{2Q}$ and $f_i^{\infty} = f_i \mathbf{1}_{(2Q)^c}$, $i=1,\ldots,\kappa$. Then there holds
that
$$
\mathscr{T}(\vec{f^r})(x)
:= \bigg(\int_{c_0 \ell(Q)}^{\infty} \int_{\Rn} \mathscr{V}_{t,y}(x,x')^2 \big| \Theta_t^\mu (\vec{f^r})(y)\big|^2 \frac{d\mu(y) dt}{t^{m+1}}\bigg)^{1/2}
\lesssim \prod_{i=1}^{\kappa} M_{\mu}(f_i)(x),
$$
where $\vec{f^r}=(f_1^{r_1},\cdots,f_{\kappa}^{r_{\kappa}})$ with $r_i \in \{0,\infty\}$ and at lest one $r_i = \infty$, and
$$
\mathscr{V}_{t,y}(x,x'):=\bigg(\frac{t}{t+|x-y|}\bigg)^{m \lambda/2} - \bigg(\frac{t}{t+|x'-y|}\bigg)^{m \lambda/2} .
$$
\end{lemma}

\begin{proof}
By symmetry, it suffices to consider the following two cases :
$$
\hbox{Case 1.}\ \vec{f^r}=(f_1^{\infty},f_2^0),\ \ \hbox{Case 2.}\ \vec{f^r} = (f_1^{\infty},f_2^{\infty}).
$$
We will treat the above cases respectively.

\vspace{0.3cm}
\noindent\textbf{Case 1.}
By Minkowski's inequality, it yields that
{\small
\begin{equation}\label{T-f}
\mathscr{T}(\vec{f^r})(x)
\leq \int_{2Q} \int_{\Rn \setminus 2Q}  \bigg(\int_{c_0 \ell(Q)}^{\infty}
\int_{\Rn} \mathscr{V}_{t,y}(x,x')^2 |s_t(y,\vec{z})|^2 \frac{d\mu}{t^m} \frac{dt}{t}\bigg)^{\frac12}
\prod_{i=1}^2 |f_i(z_i)| d\mu(z_i).
\end{equation}
}
Set
\begin{eqnarray*}
E_{1} &:=& \big\{y \in \Rn;\ |x-y| \leq t \big\}, \\
E_{2} &:=& \big\{y \in \Rn;\ |x-y| > t, |x-y| \geq |x - z_1|/2 \big\}, \\
E_{3} &:=& \big\{y \in \Rn;\ |x-y| > t, |x-y| < |x - z_1|/2 \big\}, \\
\end{eqnarray*}
Then we obtain
$$
\mathscr{T}(\vec{f})(x)
\leq \mathscr{T}_1(\vec{f})(x) + \mathscr{T}_2(\vec{f})(x) +\mathscr{T}_3(\vec{f})(x),
$$
where $\mathscr{T}_j$ corresponds to the right hand side of $(\ref{T-f})$ with the innermost integral on $E_j$, $j=1,2,3$.
Applying the mean value theorem to the function $x \mapsto (t/{(t+|x-y|)})^{m \lambda/2}$, we get $\mathscr{V}_{t,y}(x,x') \lesssim |x-x'|/t$. Note that for any $y \in E_1$
$$
t+|y-z_1| \geq |x-y| + |y-z_1| \geq |x-z_1|.
$$
Then the size condition implies that
{\small 
\begin{align*}
\mathscr{T}_{1}(\vec{f})(x)
&\lesssim \int_{2Q} \int_{\Rn \setminus 2Q} \bigg(\int_{c_0 \ell(Q)}^{\infty} \int_{|x-y| \leq t} \frac{\ell(Q)^2}{t^2} \frac{t^{2 \alpha - 2m}}{|x-z_1|^{2 m + 2 \alpha}}
\frac{d\mu}{t^m} \frac{dt}{t}\bigg)^{\frac12} \prod_{i=1}^2 |f_i(z_i)| d\mu  \\
&\lesssim \bigg(\int_{c_0 \ell(Q)}^{\infty} \frac{\ell(Q)^2}{t^{2m+2-2\alpha}} \frac{dt}{t}\bigg)^{1/2} \int_{\Rn \setminus 2 Q} \frac{|f_1(z_1)| d\mu(z_1)}{|x-z_1|^{m + \alpha}} \int_{2Q} |f_2| d\mu\\
&\lesssim \ell(Q)^{-m+\alpha} \int_{\Rn \setminus 2 Q} \frac{|f_1(z_1)|}{|x-z_1|^{m + \alpha}} d\mu(z_1) \int_{2Q} |f_2(z_2)| d\mu(z_2) \\
&\lesssim \sum_{j=1}^{\infty} 2^{-j \alpha} \frac{1}{\mu(2^{j+1}Q)} \int_{2^{j+1} Q} |f_1(z_1)| d\mu(z_1) \cdot M_{\mu}(f_2)(x)\\
&\lesssim M_{\mu}(f_1)(x) M_{\mu}(f_2)(x).
\end{align*}
}


Together with $|x-y| > t \gtrsim \ell(Q) \gtrsim |x-x'|$, the mean value theorem gives that
\begin{equation}\label{Mean}
\mathscr{V}_{t,y}(x,x')
\lesssim |x-x'| \frac{t^{m \lambda/2}}{(t+|x-y|)^{m \lambda/2 + 1}}
\lesssim \ell(Q) \frac{t^{2m + \alpha/2 - 1}}{(t+|x-y|)^{2m + \alpha/2}},
\end{equation}
where we have used $\alpha \leq m(\lambda - 4)$.

A simple calculation gives that
$$
\int_{\Rn} \prod_{i=1}^{\kappa} \frac{t^{2 \alpha}}{(t+|y-z_i|)^{2m + 2 \alpha}} \frac{d\mu(y)}{t^m}
\lesssim t^{-\kappa}.
$$
Thus, it yields that
\begin{align*}
\mathscr{T}_{2}(\vec{f})(x)
&\lesssim \int_{2Q} \int_{\Rn \setminus 2Q} \bigg(\ell(Q)^2 \int_{c_0 \ell(Q)}^{\infty} \frac{t^{4m + \alpha - 2}}{|x-z_1|^{4 m + \alpha}} \\
&\qquad\times \int_{\Rn} \prod_{i=1}^2 \frac{t^{2 \alpha}}{(t+|y-z_i|)^{2m + 2 \alpha}} \frac{d\mu(y)}{t^m} \frac{dt}{t}\bigg)^{1/2} \prod_{i=1}^2 |f_i(z_i)| d\mu(z_i)  \\
&\lesssim \bigg(\int_{c_0 \ell(Q)}^{\infty} \frac{\ell(Q)^2}{t^{2-\alpha}} \frac{dt}{t}\bigg)^{1/2} \int_{\Rn \setminus 2 Q} \frac{|f_1(z_1)|}{|x-z_1|^{2m + \alpha/2}} d\mu(z_1) \int_{2Q} |f_2| d\mu \\
&\lesssim \ell(Q)^{\alpha/2} \int_{\Rn \setminus 2 Q} \frac{|f_1(z_1)|}{|x-z_1|^{2m + \alpha/2}} d\mu(z_1) \int_{2Q} |f_2| d\mu\\
&\lesssim \sum_{j=1}^{\infty} 2^{-j \alpha/2} \prod_{i=1}^2 \frac{1}{\mu(2^{j+1}Q)} \int_{2^{j+1} Q} |f_i| d\mu 
\lesssim M_{\mu}(f_1)(x) M_{\mu}(f_2)(x).
\end{align*}
If $y \in E_3$, there holds that $|y-z_1| \geq |x-z_1| - |x-y| \geq |x-z_1|/2$. Making use of $(\ref{Mean})$ again, we have $\mathscr{T}_{3}(\vec{f})(x)$ can be controlled by  a constant times that
\begin{align*}
&\int_{2Q} \int_{\Rn \setminus 2Q} \bigg(\ell(Q)^2 \int_{c_0 \ell(Q)}^{\infty} \int_{\Rn} \frac{t^{4m + \alpha - 2}}{(t+|x-y|)^{4 m + \alpha}} 
 \\
&\qquad\qquad\times \frac{t^{2 \alpha - 2 m}}{|x-z_1|^{2m + 2 \alpha}} 
\frac{d\mu(y)}{t^m} \frac{dt}{t}\bigg)^{1/2} \prod_{i=1}^2 |f_i(z_i)| d\mu(z_i)  \\
&\lesssim \bigg(\int_{c_0 \ell(Q)}^{\infty} \frac{\ell(Q)^2}{t^{2m + 2 - 2\alpha}} \frac{dt}{t}\bigg)^{1/2}
\int_{\Rn \setminus 2 Q} \frac{|f_1(z_1)| d\mu(z_1)}{|x-z_1|^{m + \alpha}} \int_{2Q} |f_2| d\mu.
\end{align*}
The remaining arguments are the same as the term $\mathscr{T}_1$. Therefore, we deduce that
$$
\mathscr{T}_{3}(\vec{f})(x)
\lesssim M_{\mu}(f_1)(x) M_{\mu}(f_2)(x).
$$
\vspace{0.3cm}
\noindent\textbf{Case 2.}
This case can be discussed in the same manner as that of Case 1. The slight difference lies in that the domains are modified to be
\begin{eqnarray*}
E_{1} &:=& \big\{y \in \Rn;\ |x-y| \leq t \big\}, \\
E_{2} &:=& \big\{y \in \Rn;\ |x-y| > t, |x-y| \geq \min_{i=1,2}\{|x - z_i|/2\} \big\}, \\
E_{3} &:=& \big\{y \in \Rn;\ |x-y| > t, |x-y| < \min_{i=1,2}\{|x - z_i|/2\} \big\}.
\end{eqnarray*}
We omit the details here.
\end{proof}

\section{Bilinear Local $T1$ Theorem}\label{Sec-main}
This section aims to demonstrate how to deduce Theorem $\ref{Local T1}$ fromTheorems $\ref{L^p}$ and $\ref{Big piece}$.
\subsection{Proof of the main theorem}
By Theorem $\ref{Big piece}$, one can get that for every $(2,\beta)$-doubling cube $Q \subset \Rn$ with $\mathfrak{C}$-small boundary satisfying assumptions in Theorem $\ref{Local
T1}$, there exists a subset $G_Q \subset Q$ such that $\mu(G_Q) \geq \frac{1-\delta_0}{2} \mu(Q)$ and
$$
\big\| \mathbf{1}_{G_Q} g_{\lambda,\mu}^*(f_1,f_2) \big\|_{L^q(\mu)}
\lesssim \big\|f_1\big\|_{L^{q_1}(\mu)} \big\|f_2\big\|_{L^{q_2}(\mu)}
$$
for each $f_i \in L^{q_i}(\mu)$ with $\supp(f_i) \subset Q$, and $1 < q, q_1, q_2 < \infty$ with $\frac1q = \frac1{q_1} + \frac1{q_2}$. Thus, there holds that
$$
g_{\lambda,\mu \lfloor G_Q}^* : L^{q_1}(\mu \lfloor G_Q) \times L^{q_2}(\mu \lfloor G_Q) \rightarrow L^{q}(\mu \lfloor G_Q).
$$
Thereby, from Proposition $\ref{M-M}$ below, it follows that
$$
g_{\lambda}^* : \mathfrak{M}(\Rn) \times \mathfrak{M}(\Rn) \rightarrow L^{\frac12,\infty}(\mu \lfloor G_Q).
$$
Finally, making use of Theorem $\ref{L^p}$, we deduce that
$$
g_{\lambda,\mu}^* : L^{p_1}(\mu) \times L^{p_2}(\mu) \rightarrow L^{p}(\mu),
$$
for all $1 < p_1,p_2 < \infty$ and $\frac12 < p < \infty$ satisfying $\frac{1}{p}=\frac{1}{p_1}+\frac{1}{p_2}$. 
\qed

\subsection{The Endpoint Bound}\label{Sec-M-M}
\begin{proposition}\label{M-M}
Let $\lambda > 4$, $0 < \alpha \leq m(\lambda-4)$ and $\mu$ be a power bound measure.
Let $1 < p,p_1,p_2 < \infty$ with $\frac1p = \frac{1}{p_1} + \frac{1}{p_2}$.
If $g_{\lambda,\mu}^*$ is bounded from $L^{p_1}(\mu) \times L^{p_2}(\mu)$ to $L^{p}(\mu)$, then
$$
g_{\lambda}^* : \mathfrak{M}(\Rn) \times \mathfrak{M}(\Rn) \rightarrow L^{\frac12,\infty}(\mu).
$$
\end{proposition}
The key of the proof lies in the Calder\'{o}n-Zygmund decomposition of a measure, which was given in \cite{T-3}.
\begin{lemma}\label{C-Z decomposition}
Let $\mu$ be a Radon measure on $\Rn$. For any $\nu \in \mathfrak{M}(\Rn)$ with compact support and any
$\xi > 2^{n+1} ||\nu||/||\mu||$, we have:
\begin{enumerate}
\item [(a)] There exists a family of almost disjoint cubes $\{Q_i\}_i$ and a function $f \in L^1(\mu)$ such that
\begin{eqnarray}
\label{C-Z-1}|\nu|(Q_i) & > & \frac{\xi}{2^{n+1}} \mu(2Q_i);\\
\label{C-Z-2}|\nu|(\eta Q_i) & \leq & \frac{\xi}{2^{n+1}} \mu(2\eta Q_i), \ for \ any \ \eta > 2;\\
\label{C-Z-3}\nu = f \mu \ \ & {}& in \ \Rn \setminus \bigcup_i Q_i,\ \ with \ |f| \leq  \xi \ \ \mu-a.e.;
\end{eqnarray}
\item [(b)] For each $i$, let $R_i$ be a $(6, 6^{m+1})$-doubling cube concentric with $Q_i$, with $\ell(R_i) > 4\ell(Q_i)$ and denote $w_i = \mathbf{1}_{Q_i} \big/
    {\sum_{k}\mathbf{1}_{Q_k}}$. Then, there exists a family of functions $\varphi_i$ with $supp(\varphi_i) \subset R_i$, and each $\varphi_i$ with constant sign satisfying
\begin{eqnarray}
\label{C-Z-4}\int_{R_i} \varphi_i \ d\mu & = & \int_{Q_i} f w_i \ d\mu;\\
\label{C-Z-5}\sum_{i} |\varphi_i| & \lesssim & \xi;\\
\label{C-Z-6}\mu(R_i) ||\varphi_i||_{L^\infty(\mu)} & \lesssim & |\nu|(Q_i).
\end{eqnarray}
\end{enumerate}
\end{lemma}
\qed

For simplicity we may assume that $||\nu_j||=1$, $\nu_j$ has compact support for each $j$, and $\xi^{1/2} > 2^{n+1}/ ||\mu||$. Applying Lemma $\ref{C-Z decomposition}$ to the measure
$\nu_j$ at the level $\xi^{1/2}$, we have the decomposition: $\nu_j=g_j \mu + \beta_j$ with 
\begin{equation}\label{nu-g-beta}
g_j \mu = \mathbf{1}_{\Rn \setminus \bigcup Q^i_{j}} \nu + \sum_i \varphi^i_j \mu \quad\text{and}\quad 
\beta_j = \sum_i \beta^i_j := \sum_i ( w^i_j \nu - \varphi^i_j \mu),
\end{equation}
where $\{Q^i_{j}\}_i$ and $\{R^i_{j}\}_i$ are form of those in Lemma $\ref{C-Z decomposition}$. Then it is easy to get
\begin{equation}\label{g-norm}
||g_j||_{L^{\infty}(\mu)} \lesssim \xi^{1/2},\ \ ||g_j||_{L^1(\mu)} \lesssim 1, \ \text{and } ||g_j||_{L^{s}(\mu)} \lesssim \xi^{(1-1/{s})/2},\ s > 1.
\end{equation}
We write
\begin{align*}
\mathcal{I}_{11} &= \mu\big(\{x \in \Rn; g_{\lambda,\mu}^* (g_1,g_2)(x) > \xi/4 \}\big),\\ 
\mathcal{I}_{21} &= \mu\big(\{x \in \Rn; g_\lambda^* (\beta_1,g_2 \mu)(x) > \xi/4 \}\big),\\
\mathcal{I}_{12} &= \mu\big(\{x \in \Rn; g_\lambda^* (g_1 \mu,\beta_2)(x) > \xi/4 \}\big),\\ 
\mathcal{I}_{22} &= \mu\big(\{x \in \Rn; g_\lambda^* (\beta_1,\beta_2)(x) > \xi/4 \}\big),
\end{align*}
which yields that
\begin{align*}
\mu\big(\{x \in \Rn; g_\lambda^* (\nu_1,\nu_2)(x) > \xi \}\big)
\leq  \mathcal{I}_{11} + \mathcal{I}_{21} + \mathcal{I}_{12} + \mathcal{I}_{22}.
\end{align*}
We will consider the above four terms consecutively.
\subsubsection{\bf{Good/Good part.}}
By Chebychev's inequality and $(\ref{g-norm})$, we have
$$
\mathcal{I}_{11}
\lesssim \xi^{-p} \big\|g_{\lambda,\mu}^* (g_1,g_2)\big\|_{L^p(\mu)}^p
\lesssim \xi^{-p} \prod_{j=1}^2 \big\|g_j\big\|_{L^{p_j}(\mu)}^p
\lesssim \xi^{-1/2}.
$$
\subsubsection{\bf{Bad/Good and Good/Bad parts.}}
Together, the inequality $(\ref{C-Z-1})$ and the decomposition $(\ref{nu-g-beta})$ establish
\begin{align*}
\mathcal{I}_{21}
&\leq \mu\Big(\bigcup_{i}2Q_1^i\Big) + 4 \xi^{-1} \int_{\Rn \setminus \bigcup_{i}2 Q_1^i} g_\lambda^*(\beta_1,g_2 \mu)(x) d\mu(x) \\
&\lesssim \xi^{-1/2} ||\nu_1|| + \xi^{-1} \sum_i \int_{\Rn \setminus 4 R_1^i} g_\lambda^*(\beta_1^i,g_2 \mu)(x) d\mu(x) \\
&\quad + \xi^{-1} \sum_i \int_{4R_1^i \setminus 2Q_1^i} g_\lambda^*(\omega_1^i \nu_1,g_2 \mu)(x) d\mu(x) \\
&\quad    + \xi^{-1} \sum_i \int_{4R_1^i \setminus 2 Q_1^i} g_{\lambda,\mu}^*(\varphi_1^i,g_2)(x) d\mu(x).
\end{align*}
Consequently, to get the weak type bound, it suffices to conclude that for each $i$ there holds that
\begin{eqnarray}
\label{bad-good-1}\mathcal{H}_1:=\int_{\Rn \setminus 4 R_1^i} g_\lambda^*(\beta_1^i,g_2 \mu)(x) d\mu(x)
& \lesssim & \xi^{1/2} |\nu_1|(Q_1^i), \\
\label{bad-good-2}\mathcal{H}_2:=\int_{4R_1^i \setminus 2Q_1^i} g_\lambda^*(\omega_1^i \nu_1,g_2 \mu)(x) d\mu(x)
& \lesssim & \xi^{1/2} |\nu_1|(Q_1^i),\\
\label{bad-good-3}\mathcal{H}_3:=\int_{4R_1^i \setminus 2 Q_1^i} g_{\lambda,\mu}^*(\varphi_1^i,g_2 \mathbf{1}_{6R_1^i})(x) d\mu(x)
& \lesssim & \xi^{1/2} |\nu_1|(Q_1^i),\\
\label{bad-good-4}\mathcal{H}_4:=\int_{4R_1^i \setminus 2 Q_1^i} g_{\lambda,\mu}^*(\varphi_1^i,g_2 \mathbf{1}_{\Rn \setminus 6R_1^i})(x) d\mu(x)
& \lesssim & \xi^{1/2} |\nu_1|(Q_1^i)
\end{eqnarray}

First, the inequality $(\ref{bad-good-1})$ follows from Lemma $\ref{Pointwise-beta}$ below.
Combining H\"{o}lder's inequality with $(\ref{g-norm})$ and $(\ref{C-Z-6})$, it yields that
\begin{align*}
\mathcal{H}_3
&\leq \mu(4R_1^i)^{1-1/p} \big\| \mathbf{1}_{4R_1^i} g_{\lambda,\mu}^*(\varphi_1^i,g_2 \mathbf{1}_{6R_1^i}) \big\|_{L^p(\mu)} \\
& \lesssim \mu(4R_1^i)^{1-1/p} \big\|\varphi_1^i\big\|_{L^{\infty}(\mu)} \mu(4R_1^i)^{1/{p_1}}  \big\|g_2\big\|_{L^{\infty}(\mu)} \mu(4R_1^i)^{1/{p_2}} \\
& \lesssim \xi^{1/2} \mu(R_1^i) \big\|\varphi_1^i\big\|_{L^{\infty}(\mu)}
\lesssim  \xi^{1/2} |\nu_1|(Q_1^i).
\end{align*}
This shows the inequality $(\ref{bad-good-3})$

Secondly, to gain the inequality $(\ref{bad-good-4})$, we treat the contribution of the kernel. The size condition gives that
\begin{align*}
\big| \Theta_t^{\mu}(\varphi_1^i,g_2 \mathbf{1}_{\Rn \setminus 6R_1^i})(y) \big|
\lesssim t^{-m} \mu(R_1^i) \|\varphi_1^i\|_{L^{\infty}(\mu)}
\int_{\Rn \setminus 6 R_1^i} \frac{t^\alpha  \|g_2\|_{L^{\infty}(\mu)} d\mu(z_2)}{(t+|y-z_2|)^{m+\alpha}}.
\end{align*}
It follows from $(\ref{U(f)})$ that for every $x \in 4R_1^i$
\begin{align*}
&\bigg(\int_{\Rn} \Big(\frac{t}{t+|x-y|}\Big)^{m \lambda} \big| \theta_t^{\mu}(\varphi_1^i,g_2 \mathbf{1}_{\Rn \setminus 6R_1^i})(y) \big|^2 \frac{d\mu(y)}{t^m} \bigg)^{1/2} \\
&\lesssim \big\|\varphi_1^i\big\|_{L^{\infty}(\mu)} \big\| g_2 \big\|_{L^{\infty}(\mu)} \min\Big\{t^{\alpha} \ell(R_1^i)^{-\alpha}, t^{-m} \mu(R_1^i) \Big\},
\end{align*}
which indicates that
\begin{equation}\label{Rn-6R}
g_{\lambda,\mu}^*(\varphi_1^i,g_2 \mathbf{1}_{\Rn \setminus 6R_1^i})(x)
\lesssim \xi^{1/2} \big\|\varphi_1^i\big\|_{L^{\infty}(\mu)}.
\end{equation}
Therefore, it holds that
$$
\mathcal{H}_4
\lesssim \xi^{1/2} \mu(R_1^i) \big\|\varphi_1^i\big\|_{L^{\infty}(\mu)}
\lesssim  \xi^{1/2} |\nu_1|(Q_1^i).
$$

Finally, we prove the inequality ''''$(\ref{bad-good-2})$, it is sufficient to show the following estimate
\begin{equation}\label{bad-4}
g_\lambda^*(w_1^i \nu_1,g_2 \mu)(x)
\lesssim \xi^{1/2} \frac{|\nu_1|(Q_1^i)}{|x-c_{Q_1^i}|^m} , \ x \in 4R_1^i \setminus 2Q_1^i.
\end{equation}
Actually,
\begin{equation}\label{R-Q}
\int_{4R_i \setminus 2Q_i} \frac{d\mu(x)}{|x-c_{Q_i}|^m}
\leq \bigg(\int_{4R_i \setminus R_i} +  \int_{R_i \setminus 6Q_i} + \int_{6 Q_i \setminus Q_i}\bigg) \frac{d\mu(x)}{|x-c_{Q_i}|^m}.
\end{equation}
It is easy to see that
$$\bigg(\int_{4R_i \setminus R_i} + \int_{6 Q_i \setminus Q_i}\bigg) \frac{d\mu(x)}{|x-c_{Q_i}|^m}
\lesssim \frac{\mu(4R_i)}{\ell(R_i)^m} + \frac{\mu(6Q_i)}{\ell(Q_i)^m} \lesssim 1.$$
Moreover, there are no $(6,6^{m+1})$-doubling cubes of the form $6^kQ_i$ such that $6Q_i \subsetneq 6^k Q_i \subsetneq R_i$.
Let $ N_i := \min\{ k; R_i \subset 6^k \cdot 6Q_i \}$.
Hence, $$ \mu(6 \cdot 6^k Q_i) > 6^{m+1} \mu(6^k Q_i), \ k=1,\ldots,N_i.$$
and hence,
$$ \mu(6^{N_i}\cdot6Q_i) > 6^{(m+1)(N_i-k)} \mu(6^k Q_i).$$
Therefore,
\begin{align*}
\int_{R_i \setminus 6Q_i} \frac{d\mu(x)}{|x-c_{Q_i}|^m}
&\leq \sum_{k=1}^{N_i} \int_{6^{k+1}Q_i \setminus 6^k Q_i} \frac{d\mu(x)}{|x-c_{Q_i}|^m}
\\
&\lesssim \sum_{k=1}^{N_i} \frac{\mu(6^{k+1}Q_i)}{\ell(6^k Q_i)^m}
\lesssim  \sum_{k=1}^{N_i} 6^{k-N_i} \frac{\mu(6^{N_i+1}Q_i)}{\ell(6^{N_i+1} Q_i)^m}
\lesssim 1.
\end{align*}

Now let us show $(\ref{bad-4})$. The size condition implies that
\begin{align*}
\Theta_t(w_1^i \nu_1,g_2 \mu)(y)
&\lesssim \int_{Q_1^i} \frac{t^{\alpha} d|\nu_1|(z_1)}{(t+|y-z_1|)^{m+\alpha}}  
\int_{\Rn} \frac{t^{\alpha} ||g_2||_{L^{\infty}(\mu)} d\mu(z_2)}{(t+|y-z_2|)^{m+\alpha}}   
\\
&\lesssim \xi^{1/2} \int_{Q_1^i} \frac{t^{\alpha}}{(t+|y-z|)^{m+\alpha}} d|\nu_1|(z).
\end{align*}

Together with Lemma $\ref{U(f)}$, this gives that for $x \in 4R_1^i \setminus 2Q_1^i$
\begin{align*}
&\bigg(\int_{\Rn} \Big(\frac{t}{t+|x-y|}\Big)^{m\lambda} |\Theta_t (w_1^i \nu_1,g_2 \mu)(y)|^2  \frac{d\mu(y)}{t^m}\bigg)^{1/2} \\
&\lesssim  \xi^{1/2} \int_{Q_1^i} \frac{t^{\alpha}}{(t+|x-z|)^{m+\alpha}} d|\nu_1|(z)
\\
&\lesssim  \xi^{1/2}  \frac{t^{\alpha}}{(t+|x-c_{Q_1^i}|)^{m+\alpha}} |\nu_1|(Q_1^i).
\end{align*}
Therefore, the desired result can be obtained
\begin{align*}
g_{\lambda}^*(w_1^i \nu_1,g_2 \mu)(x)
&\lesssim \xi^{1/2} |\nu_1|(Q_1^i) \bigg(\int_{|x-c_{Q_1^i}|}^{\infty} \frac{1}{t^{2 m}} \frac{dt}{t}\bigg)^{1/2} \\
&\quad +  \xi^{1/2} |\nu_1|(Q_1^i) \bigg(\int_{0}^{|x-c_{Q_1^i}|} \frac{t^{2 \alpha}}{|x-c_{Q_1^i}|^{2(m+\alpha)}} \frac{dt}{t}\bigg)^{1/2}\\
&\lesssim \xi^{1/2} \frac{|\nu_1|(Q_1^i)}{|x-c_{Q_1^i}|^m}.
\end{align*}
This completes the proof. 
\qed
\subsubsection{\bf{Bad/Bad part.}}
Write $\Q:=\bigcup_{i,j} Q_j^i$. Since $\mu(\Q) \lesssim \xi^{-1/2}$, it suffices to bound
$$
\mu\big(\{x \in \Rn \setminus \Q; g_\lambda^* (\beta_1,\beta_2)(x) > \xi \}\big) \lesssim \xi^{-1/2}.
$$
By symmetry and sub-linearity, it is enough to show separately
\begin{eqnarray}
\label{K-1}\mathcal{K}_1
&:=& \mu\Big(\Big\{x \in \Rn \setminus \Q; \sum_{i} \sum_{j \in \Lambda_i} \mathbf{1}_{\Rn \setminus 4 R_1^i} g_\lambda^* (\beta_1^i,\beta_2^j)(x) > \xi \Big\}\Big)
\lesssim \xi^{-1/2}, \\
\label{K-2}\mathcal{K}_2
&:=& \mu\Big(\Big\{x \in \Rn \setminus \Q; \sum_{i} \mathbf{1}_{4 R_1^i} g_\lambda^* \Big(\beta_1^i,\sum_{j \in \Lambda_i} \beta_2^j\Big)(x) > \xi \Big\}\Big)
\lesssim \xi^{-1/2},
\end{eqnarray}
where $\Lambda_i = \big\{j;\ \ell(R_1^i) \leq \ell(R_2^j) \big\}$.

\vspace{0.3cm}
\noindent\textbf{$\bullet$ Case 1.}
By Chebychev's inequality, we have
\begin{align*}
\mathcal{K}_1
&\leq \xi^{-1/2} \int_{\Rn \setminus \Q} \Big(\sum_{i} \sum_{j \in \Lambda_i} \mathbf{1}_{\Rn \setminus (4 R_1^i \cup 4 R_2^j)}(x)
g_\lambda^* (\beta_1^i,\beta_2^j)(x)\Big)^{-1/2} d\mu(x) \\
&\quad + \xi^{-1/2} \int_{\Rn \setminus \Q} \Big(\sum_{i} \sum_{j} \mathbf{1}_{4 R_2^j \setminus 4 R_1^i}(x)
g_\lambda^* (\beta_1^i,\varphi_2^j \mu)(x)\Big)^{-\frac12} d\mu(x)\\
&\quad + \xi^{-1/2} \int_{\Rn \setminus \Q} \Big(\sum_{i} \sum_{j} \mathbf{1}_{4 R_2^j \setminus 4 R_1^i}(x)
g_\lambda^* (\beta_1^i,w_2^j \nu_2)(x)\Big)^{-\frac12} d\mu(x)\\
&:= \mathcal{K}_{11} + \mathcal{K}_{12} + \mathcal{K}_{13}.
\end{align*}
It follows from the pointwise control $(\ref{point-b-b})$ that
\begin{align*}
\mathcal{K}_{11}
&\lesssim \xi^{-1/2} \int_{\Rn} \bigg(\sum_{i, j} \mathbf{1}_{\Rn \setminus (4 R_1^i \cup 4 R_2^j)}
          \frac{\ell(R_1^i)^{\alpha/4} |\nu_1|(Q_1^i)}{|x-c_{R_1^i}|^{m+\alpha/4}}
          \frac{\ell(R_2^j)^{\alpha/4} |\nu_2|(Q_2^j)}{|x-c_{R_2^j}|^{m+\alpha/4}} \bigg)^{\frac12} d\mu\\
&\lesssim \xi^{-1/2} \bigg(\sum_{i} |\nu_1|(Q_1^i) \int_{\Rn \setminus 4 R_1^i} \frac{\ell(R_1^i)^{\alpha/4}}{|x-c_{R_1^i}|^{m+\alpha/4}} d\mu(x) \bigg)^{\frac12} \\
&\quad\quad\times     \bigg(\sum_{j} |\nu_2|(Q_2^j) \int_{\Rn \setminus 4 R_2^j} \frac{\ell(R_2^j)^{\alpha/4}}{|x-c_{R_2^j}|^{m+\alpha/4}} d\mu(x) \bigg)^{\frac12}  \\
&\lesssim \xi^{-1/2} \bigg(\sum_{i} |\nu_1|(Q_1^i) \bigg)^{1/2} \bigg(\sum_{j} |\nu_2|(Q_2^j) \bigg)^{\frac12}
\leq \xi^{-1/2}.
\end{align*}
The second term can be bounded as follow. Applying $(\ref{point-b-var})$, $(\ref{C-Z-6})$ and the doubling property of $R_2^j$, we deduce that
\begin{align*}
\mathcal{K}_{12}
&\lesssim \xi^{-1/2} \int_{\Rn} \bigg(\sum_{i,j } \mathbf{1}_{4 R_2^j \setminus 4 R_1^i}(x)
          \frac{\ell(R_1^i)^{\alpha/4} |\nu_1|(Q_1^i)}{|x-c_{R_1^i}|^{m+\alpha/4}} \big\| \varphi_2^j \big\|_{L^{\infty}} \bigg)^{\frac12} d\mu(x) \\
&\lesssim \xi^{-1/2} \bigg(\sum_{i} |\nu_1|(Q_1^i) \int_{\Rn \setminus 4 R_1^i} \frac{\ell(R_1^i)^{\alpha/4} d\mu}{|x-c_{R_1^i}|^{m+\alpha/4}} \bigg)^{\frac12}
 \bigg(\sum_{j} \mu(R_2^j) \big\| \varphi_2^j \big\|_{L^{\infty}} \bigg)^{\frac12}  \\
&\lesssim \xi^{-1/2}.
\end{align*}
By $(\ref{point-b-w})$ and $(\ref{R-Q})$, it yields that
\begin{align*}
\mathcal{K}_{13}
&\lesssim \xi^{-1/2} \int_{\Rn \setminus \mathcal{Q}} \bigg(\sum_{i} \sum_{j} \mathbf{1}_{4 R_2^j \setminus 4 R_1^i}(x)
          \frac{\ell(R_1^i)^{\alpha/2} |\nu_1|(Q_1^i)}{|x-c_{R_1^i}|^{m+\alpha/2}}
          \frac{ |\nu_2|(Q_2^j)}{|x-c_{R_2^j}|^{m}} \bigg)^{\frac12} d\mu \\
&\lesssim \xi^{-1/2} \bigg(\sum_{i} |\nu_1|(Q_1^i) \int_{\Rn \setminus 4 R_1^i} \frac{\ell(R_1^i)^{\alpha/2}}{|x-c_{R_1^i}|^{m+\alpha/2}} d\mu(x) \bigg)^{\frac12} \\
&\quad\quad\times     \bigg(\sum_{j} |\nu_2|(Q_2^j) \int_{4 R_2^j \setminus 2 Q_2^j} \frac{d\mu(x)}{|x-c_{R_2^j}|^{m}}  \bigg)^{\frac12}  \\
&\lesssim  \xi^{-1/2}.
\end{align*}
This shows $(\ref{K-1})$.
\qed

\vspace{0.3cm}
\noindent\textbf{$\bullet$ Case 2.}
The decompositions of $\beta_1^i$ and $\beta_2^j$ indicate that
$$
\mathcal{K}_2
\leq \mathcal{K}_{21} + \mathcal{K}_{22} + \mathcal{K}_{22} + \mathcal{K}_{23} + \mathcal{K}_{24} + \mathcal{K}_{25},
$$
where
\begin{eqnarray*}
\mathcal{K}_{21}
&:=& \xi^{-1/2} \int_{\Rn \setminus \Q} \Big(\sum_{i} \sum_{j: 4 R_1^i \cap 4 R_2^j = \emptyset}
\mathbf{1}_{4 R_1^i}(x) g_\lambda^* (\beta_1^i,\beta_2^j)(x)\Big)^{-1/2} d\mu(x), \\
\mathcal{K}_{22}
&:=& \mu\Big(\Big\{x \in \Rn \setminus \Q; \sum_{i} \mathbf{1}_{4 R_1^i}(x)
g_{\lambda,\mu}^* \Big(\varphi_1^i, \sum_{\substack{j \in \Lambda_i \\ 4 R_1^i \cap 4 R_2^j \neq \emptyset}}\varphi_2^j \Big)(x) > \xi \Big\}\Big) \\
\mathcal{K}_{23}
&:=& \xi^{-1/2} \int_{\Rn \setminus \Q} \Big(\sum_{i} \sum_{\substack{j \in \Lambda_i \\ 4 R_1^i \cap 4 R_2^j \neq \emptyset}}
\mathbf{1}_{4 R_1^i}(x) g_\lambda^* (\varphi_1^i \mu,w_2^j \nu_2)(x) \Big)^{-1/2} d\mu(x), \\
\mathcal{K}_{24}
&:=& \xi^{-1/2} \int_{\Rn \setminus \Q} \Big(\sum_{i} \sum_{\substack{j \in \Lambda_i \\ 4 R_1^i \cap 4 R_2^j \neq \emptyset}}
\mathbf{1}_{4 R_1^i}(x) g_\lambda^* (w_1^i \nu_1,\varphi_2^j \mu)(x) \Big)^{-1/2} d\mu(x), \\
\mathcal{K}_{25}
&:=& \xi^{-1/2} \int_{\Rn \setminus \Q} \Big(\sum_{i} \sum_{\substack{j \in \Lambda_i \\ 4 R_1^i \cap 4 R_2^j \neq \emptyset}}
\mathbf{1}_{4 R_1^i}(x) g_\lambda^* (w_1^i \nu_1,w_2^j \nu_2)(x)\Big)^{-1/2} d\mu(x)¡£
\end{eqnarray*}

We first discuss the term $\mathcal{K}_{22}$. The splitting $\varphi_2^j = \varphi_2^j \mathbf{1}_{6 R_1^i} +¡¡\varphi_2^j \mathbf{1}_{\Rn \setminus 6 R_1^i}$ yields two terms denoted
by
$\mathcal{K}_{22}'$ and $\mathcal{K}_{22}''$. The boundedness of $g_{\lambda,\mu}^* : L^{p_1}(\mu) \times L^{p_2}(\mu) \rightarrow L^{p}(\mu)$ implies that
\begin{align*}
\mathcal{K}_{22}'
&\lesssim \xi^{-1} \sum_{i} \mu(R_1^i)^{1-1/p} \Big\| g_{\lambda,\mu}^* \Big(\varphi_1^i,\sum_{j}\varphi_2^j \mathbf{1}_{6 R_1^i}\Big)\Big\|_{L^p(\mu)} \\
&\lesssim \xi^{-1} \sum_{i} \mu(R_1^i)^{1-1/p} ||\varphi_1^i||_{L^{\infty}(\mu)}  \mu(R_1^i)^{1/{p_1}} \Big\| \mathbf{1}_{6 R_1^i} \sum_{j}\varphi_2^j \Big\|_{L^{p_2}(\mu)} \\
&\lesssim \xi^{-1/2} \sum_{i} \mu(R_1^i) ||\varphi_1^i||_{L^{\infty}(\mu)}
\lesssim \xi^{-1/2},
\end{align*}
where the inequalities $(\ref{C-Z-5})$ and $(\ref{C-Z-6})$ were used.
In order to bound $\mathcal{K}_{22}''$, we follow exactly the same scheme of proof of the inequality $(\ref{Rn-6R})$ with slight modifications
(replacing $||g||_{L^{\infty}(\mu)} \lesssim \xi^{1/2}$ by $\sum_{j}|\varphi_2^j| \lesssim \xi^{1/2}$). Then we get for any $x \in 4 R_1^i$
$$
g_{\lambda,\mu}^* \Big(\varphi_1^i,\sum_{j}\varphi_2^j \mathbf{1}_{\Rn \setminus 6 R_1^i}\Big)(x)
\lesssim \xi^{1/2} \big\|\varphi_1^i\big\|_{L^{\infty}(\mu)},
$$
which indicates that
\begin{align*}
\mathcal{K}_{22}''
&\lesssim \xi^{-1} \sum_{i} \int_{4 R_1^i} g_{\lambda,\mu}^* \Big(\varphi_1^i,\sum_{j}\varphi_2^j \mathbf{1}_{\Rn \setminus 6 R_1^i}\Big)(x) d\mu(x) \\
&\lesssim \xi^{-1/2} \sum_{i} \mu(R_1^i) ||\varphi_1^i||_{L^{\infty}(\mu)}
\lesssim \xi^{-1/2}.
\end{align*}

If $4 R_1^i \cap 4 R_2^j = \emptyset$, then $\mathbf{1}_{4 R_1^i} = \mathbf{1}_{4 R_1^i \setminus 4 R_2^j}$. From the fact $\beta_1^i = \varphi_1^i \mu + w_1^i \nu_1$, it follows that
$\mathcal{K}_{21}$ is dominated by two terms, which are symmetric with $\mathcal{K}_{12}$ and $\mathcal{K}_{13}$ respectively. Hence, there holds that
$\mathcal{K}_{21} \lesssim \xi^{-1/2}$.

Almost similar calculations as $(\ref{bad-4})$ provide
$$
g_\lambda^* (\varphi_1^i \mu,w_2^j \nu_2)(x)
\lesssim \big\| \varphi_1^i \big\|_{L^{\infty}(\mu)}
          \frac{|\nu_2|(Q_2^j)}{|x-c_{R_2^j}|^{m}}, \ x \in \Rn \setminus 2 Q_2^j.
$$
Combining the estimates for $\mathcal{K}_{12}$ with these for $\mathcal{K}_{13}$, we gain that $\mathcal{K}_{23} \lesssim \xi^{-1/2}$.
Symmetrically, we have $\mathcal{K}_{24} \lesssim \xi^{-1/2}$. In addition, making use of $(\ref{point-w-w})$, we similarly deduce that
$\mathcal{K}_{25} \lesssim \xi^{-1/2}$.

So far, we have proved Proposition $\ref{M-M}$.
\qed

\vspace{0.2cm}
The remainder of this section is devoted to demonstrating some lemmas we used above.
\begin{lemma}\label{Pointwise-beta}
The following point-wise estimates hold for any $x \in \Rn \setminus 4R_i$
\begin{eqnarray}
\label{point-b-g}g_\lambda^*(\beta_1^i,g_2 \mu)(x)
&\lesssim& \xi^{1/2} \frac{\ell(R_1^i)^{\alpha/2}}{|x-c_{R_1^i}|^{m+\alpha/2}} \big\| \beta_1^i \big\|,  \\
\label{point-b-var}g_\lambda^*(\beta_1^i,\varphi_2^j \mu)(x)
&\lesssim&  \big\| \varphi_2^j \big\|_{L^{\infty}(\mu)} \frac{\ell(R_1^i)^{\alpha/2} }{|x-c_{R_1^i}|^{m+\alpha/2}} |\nu_1|(Q_1^i).
\end{eqnarray}
\end{lemma}

\begin{proof}
We here only show the first inequality, since the second one can be obtained similarly.
By splitting the domain, it suffices to estimate the following three terms :
\begin{eqnarray*}
&{}&\Gamma_1(x) := \bigg(\iint_{\Rn \times (0,\ell(R_i))} \Big(\frac{t}{t+|x-y|}\Big)^{m\lambda} |\Theta_t (\beta_1^i,g_2 \mu)(y)|^2 \frac{d\mu(y) dt}{t^{m+1}}\bigg)^{\frac12}, \\
&{}&\Gamma_2(x) := \bigg(\iint_{\Rn \times [\ell(R_i),|x-c_{R_i}|]} \Big(\frac{t}{t+|x-y|}\Big)^{m\lambda} |\Theta_t (\beta_1^i,g_2 \mu)(y)|^2 \frac{d\mu(y) dt}{t^{m+1}} \bigg)^{\frac12},
\\
&{}&\Gamma_3(x) := \bigg(\iint_{\Rn \times (|x-c_{R_i}|,+\infty)} \Big(\frac{t}{t+|x-y|}\Big)^{m\lambda} |\Theta_t (\beta_1^i,g_2 \mu)(y)|^2 \frac{d\mu(y) dt}{t^{m+1}}\bigg)^{\frac12}.
\end{eqnarray*}
Applying the size condition and the inequality $(\ref{U-L})$, we conclude that for $x \in \Rn \setminus 4R_1^i$ and $t < \ell(R_1^i)$
\begin{align*}
&\bigg(\int_{\Rn} \Big(\frac{t}{t+|x-y|}\Big)^{m\lambda} |\Theta_t (\beta_1^i,g_2 \mu)(y)|^2 \frac{d\mu(y)}{t^{m}}\bigg)^{1/2} \\
&\lesssim ||g_2||_{L^{\infty}(\mu)} \int_{R_1^i} \frac{t^{\alpha/4}}{(t+|x-z|)^{m+\alpha/4}} d|\beta_i|(z) 
\lesssim \xi^{1/2} \frac{t^{\alpha/4}}{|x-c_{R_1^i}|^{m+\alpha/4}} || \beta_i ||.
\end{align*}
Hence, it immediately yields that
\begin{align*}
\Gamma_1(x)
\lesssim \xi^{1/2} \frac{\ell(R_1^i)^{\alpha/4}}{|x-c_{R_1^i}|^{m+\alpha/4}} || \beta_i ||.
\end{align*}

An application of the vanishing property $\beta_i(R_i)=0$ and H\"{o}lder condition implies that
\begin{align}\label{beta-g}
|\Theta_t (\beta_1^i,g_2 \mu)(y)|
&= \bigg| \int_{\Rn} \int_{R_i} (s_t(y,z_1,z_2)-s_t(y,c_{R_i},z_2)) g_2(z_2) d\beta_i(z) d\mu(z_2)\bigg| 
\nonumber \\
&\lesssim ||g_2||_{L^{\infty}(\mu)} \int_{R_i} \frac{|z-c_{R_i}|^\alpha}{(t+|y-z_1|)^{m+\alpha}} d|\beta_i|(z_1) .
\end{align}
If $t \geq \ell(R_1^i)$, it holds
\begin{equation*}
|\Theta_t (\beta_1^i,g_2 \mu)(y)|
\lesssim \xi^{-1/2}\ell(R_i)^{\alpha/2} \int_{R_i} \frac{t^{\alpha/2}}{(t+|y-z|)^{m+\alpha}} d|\beta_i|(z).
\end{equation*}
Together with $(\ref{U-L})$, this yields that
\begin{align*}
&\int_{\Rn} |\Theta_t (\beta_1^i,g_2 \mu)(y)|^2 \Big(\frac{t}{t+|x-y|}\Big)^{m\lambda} \frac{d\mu(y)}{t^m}
\lesssim \xi^{1/2} \frac{t^\alpha \ell(R_i)^\alpha}{|x-c_{R_i}|^{2m+2\alpha}} || \beta_i ||^2.
\end{align*}
Accordingly, we get
\begin{align*}
\Gamma_2(x)
\lesssim  \xi^{1/2} \bigg( \int_{\ell(R_i)}^{|x-c_{R_i}|} \frac{t^{\alpha} \ell(R_i)^\alpha  ||\beta_i||^2}{|x-c_{R_i}|^{2m+2\alpha}} \frac{dt}{t} \bigg)^{1/2}
\simeq \xi^{1/2} \frac{\ell(R_i)^{\alpha/2}}{|x-c_{R_i}|^{m+\alpha/2}} || \beta_i ||.
\end{align*}

On the other hand, the inequality $(\ref{beta-g})$ gives that
\begin{align*}
|\Theta_t (\beta_1^i,g_2 \mu)(y)|
\lesssim \xi^{1/2} \frac{\ell(R_1^i)^\alpha}{t^{m+\alpha}} || \beta_1^i ||.
\end{align*}
Finally, we deduce that
\begin{align*}
\Gamma_3(x)
&\lesssim || \beta_1^i || \bigg(\int_{|x-c_{R_1^i}|}^{\infty} \frac{\ell(R_1^i)^{2\alpha}}{t^{2m+2\alpha}}
\int_{\Rn}\Big(\frac{t}{t+|x-y|}\Big)^{m\lambda} \frac{d\mu(y)}{t^m} \frac{dt}{t}  \bigg)^{1/2} \\
&\lesssim || \beta_1^i || \bigg(\int_{|x-c_{R_1^i}|}^{\infty} \frac{\ell(R_1^i)^{2\alpha}}{t^{2m+2\alpha}} \frac{dt}{t} \bigg)^{1/2}
\simeq || \beta_1^i || \frac{\ell(R_1^i)^\alpha}{|x-c_{R_1^i}|^{m+\alpha}}.
\end{align*}
This finishes the proof of $(\ref{point-b-g})$.
\end{proof}
\begin{lemma}\label{Pointwise-w}
We have the following point-wise dominations.
\begin{align}
\label{point-b-b} &g_\lambda^*(\beta_1^i,\beta_2^j)(x)
\lesssim  \frac{\ell(R_1^i)^{\alpha/4} |\nu_1|(Q_1^i)}{|x-c_{R_1^i}|^{m+\alpha/4}}
          \frac{\ell(R_2^j)^{\alpha/4} |\nu_2|(Q_2^j)}{|x-c_{R_2^j}|^{m+\alpha/4}} , \ x \in \Rn \setminus (4 R_1^i \cup 4 R_2^j) \\
\label{point-b-w} &g_\lambda^*(\beta_1^i,w_2^j \nu_2)(x)
\lesssim  \frac{\ell(R_1^i)^{\alpha/4} |\nu_1|(Q_1^i)}{|x-c_{R_1^i}|^{m+\alpha/4}}
          \frac{ |\nu_2|(Q_2^j)}{|x-c_{R_2^j}|^{m}} , \ x \in  4 R_2^j \setminus (4 R_1^i \cup 4 Q_2^j), \\
\label{point-w-w} &g_\lambda^* (w_1^i \nu_1,w_2^j \nu_2)(x)
\lesssim  \frac{|\nu_1|(Q_1^i)}{|x-c_{R_1^i}|^{m}}
          \frac{|\nu_2|(Q_2^j)}{|x-c_{R_2^j}|^{m}}, \ x \in \Rn \setminus (2 Q_1^i \cup 2 Q_2^j).
\end{align}
\end{lemma}
\begin{proof}
To dominate the term $g_\lambda^*(\beta_1^i,\beta_2^j)(x)$, we first analyze the contribution of $\Theta_t(\beta_1^i,\beta_2^j)(y)$.
Note that $\beta_1^i(R_1^i)=0$ and $\beta_2^j(R_2^j)=0$. Applying size condition and the vanishing property with respect to $z_1$ and $z_2$ respectively,
we obtain that
\begin{equation}\aligned\label{Theta-b-b}
\big|\Theta_t(\beta_1^i,\beta_2^j)(y)\big|
&\lesssim \min\bigg\{1,\frac{\ell(R_1^i)^{\alpha}}{t^{\alpha}}, \frac{\ell(R_2^j)^{\alpha}}{t^{\alpha}} \bigg\}
\int_{R_1^i} \frac{t^{\alpha} d|\beta_1^i|(z_1)}{(t+|y-z_1|)^{m+\alpha}}  \\
&\quad \times \int_{R_2^j} \frac{t^{\alpha}}{(t+|y-z_2|)^{m+\alpha}} d|\beta_2^j|(z_2).
\endaligned
\end{equation}
It follows from Lemma $\ref{U(f)}$ that
\begin{align*}
g_\lambda^*(\beta_1^i,\beta_2^j)(x)
&\lesssim \bigg\{\int_{0}^{\infty} \min\bigg\{1,\frac{\ell(R_1^i)^{\alpha}}{t^{\alpha}}, \frac{\ell(R_2^j)^{\alpha}}{t^{\alpha}} \bigg\}^2 \\
&\qquad\times \bigg(\int_{R_1^i} \frac{t^{\alpha/4}}{(t+|x-z_1|)^{m+\alpha/4}} d|\beta_1^i|(z_1)\bigg)^2 \\
&\qquad \times\bigg( \int_{R_2^j} \frac{t^{\alpha/4}}{(t+|x-z_2|)^{m+\alpha/4}} d|\beta_2^j|(z_2) \bigg)^2 \frac{dt}{t}\bigg\}^{1/2} \\
&\leq \bigg\{\int_{0}^{\min\{\ell(R_1^i),\ell(R_2^j)\}} \frac{t^{\alpha/2} \ ||\beta_1^i||^2 }{|x-c_{R_1^i}|^{2m+\alpha/2}}
\frac{t^{\alpha/2} \ ||\beta_2^j||^2}{|x-c_{R_2^j}|^{2m+\alpha/2}}  \frac{dt}{t}\bigg\}^{1/2} \\
&\quad + \bigg\{\int_{\min\{\ell(R_1^i),\ell(R_2^j)\}}^{\infty} A(t) \frac{t^{\alpha/2} \ ||\beta_1^i||^2 }{|x-c_{R_1^i}|^{2m+\alpha/2}}
\frac{t^{\alpha/2} \ ||\beta_2^j||^2}{|x-c_{R_2^j}|^{2m+\alpha/2}}  \frac{dt}{t}\bigg\}^{1/2} \\
&\lesssim \frac{\ell(R_1^i)^{\alpha/4} |\nu_1|(Q_1^i)}{|x-c_{R_1^i}|^{m+\alpha/4}}
          \frac{\ell(R_2^j)^{\alpha/4} |\nu_2|(Q_2^j)}{|x-c_{R_2^j}|^{m+\alpha/4}}, 
\end{align*}
where $A(t)=\min\big\{\frac{\ell(R_1^i)^{\alpha}}{t^{\alpha}}, \frac{\ell(R_2^j)^{\alpha}}{t^{\alpha}} \big\}^2$. 
For the term $g_\lambda^*(\beta_1^i,w_2^j \nu_2)(x)$, it suffices to use size condition and the H\"{o}lder condition with respect to $z_1$ to get a similar bound as
$(\ref{Theta-b-b})$.
The last one can be handled using size condition. The rest of calculations are easy.
\end{proof}

\section{Non-homogeneous Good Lambda Method}\label{Sec-lambda}
The proof of Theorem $\ref{L^p}$ mainly consists of the following good lambda inequality.
\begin{lemma}\label{good-lambda}
For any $\epsilon > 0$, there exists $\delta=\delta(\epsilon)>0$ such that
\begin{align*}
&\mu \big(\big\{x;g_{\lambda,\mu,t_0}^*(f_1,f_2)(x) > (1+\epsilon)\xi, M_{\mu}f_1(x) M_{\mu}f_2(x) \leq \delta \xi  \big\}\big) \\
&\leq \Big( 1 - \frac{\theta}{16 \rho_0} \Big) \mu\big(\{x; g_{\lambda,\mu,t_0}^*(f_1,f_2)(x) > \xi\}\big),
\end{align*}
for any $\xi>0$ and every compactly supported and bounded $f_i \in L^{p_i}(\mu)$.
\end{lemma}

\noindent\textbf{Proof of Theorem $\ref{L^p}$.}
Without loss of generality, we may assume that $f_i \in L^{p_i}(\mu)$ has a compact support and is bounded.
It suffices to show that $g_{\lambda,\mu,t_0}^*$ is bounded on $L^p(\mu)$ uniformly in $t_0$.
The inequality $(\ref{Prior})$ gives a prior bound, $\big\|g_{\lambda,\mu,t_0}^* (f_1,f_2)\big\|_{L^p(\mu)} < \infty$.

Lemma $\ref{good-lambda}$ gives that
\begin{equation*}\aligned\label{theta/4}
&\mu \big(\{x;  g_{\lambda,\mu,t_0}^* (f_1,f_2)(x) > (1+\epsilon)\xi\}\big) \\
&\leq \Big( 1 - \frac{\theta}{16 \rho_0} \Big) \mu\big(\{x; g_{\lambda,\mu,t_0}^*(f_1,f_2)(x) > \xi\}\big)
+\big(\{x; M_{\mu}f_1(x) M_{\mu}f_2(x)> \delta \xi \}\big).
\endaligned
\end{equation*}
Note that
\begin{equation*}\label{distribution}
\big\| f \big\|_{L^r(\mu)}^r = r \int_{0}^\infty t^{r-1} \mu(\{x; |f(x)| > t \}) dt.
\end{equation*}
Consequently, it follows that
\begin{align*}
&\big\| g_{\lambda,\mu,t_0}^* (f_1,f_2) \big\|_{L^p(\mu)}^p
=(1+\epsilon)^p p \int_{0}^\infty \xi^{p-1} \mu(\{x; g_{\lambda,\mu,t_0}^*(f_1,f_2)(x)| > (1+\epsilon)\xi \}) d\xi \\
&\leq (1+\epsilon)^p \Big( 1 - \frac{\theta}{16 \rho_0} \Big) p \int_{0}^\infty \xi^{p-1} \mu\big(\{x; g_{\lambda,\mu,t_0}^*(f_1,f_2)(x) > \xi\}\big) d\xi \\
&\quad +(1+\epsilon)^p p \int_{0}^\infty \xi^{p-1} \mu \big(\{x; M_{\mu}f_1(x) M_{\mu}f_2(x) > \delta \xi \}\big) d\xi \\
&= (1+\epsilon)^p \Big( 1 - \frac{\theta}{16 \rho_0} \Big) \big\| g_{\lambda,\mu,t_0}^* (f_1,f_2) \big\|_{L^p(\mu)}^p
+ (1+\epsilon)^p \delta^{-p} \big\| M_{\mu}f_1 \cdot M_{\mu}f_2 \big\|_{L^p(\mu)}^p.
\end{align*}
Since $\big\|g_{\lambda,\mu,t_0}^* (f_1,f_2)\big\|_{L^p(\mu)} < \infty$, taking $\epsilon>0$ small enough, we deduce that
$$
\big\|g_{\lambda,\mu,t_0}^* (f_1,f_2)\big\|_{L^p(\mu)}
\lesssim_{\theta,\delta} \big\| M_{\mu}f_1 \big\|_{L^{p_1}(\mu)}  \big\| M_{\mu}f_2 \big\|_{L^{p_2}(\mu)}
\lesssim_{\theta,\delta}  \big\| f_1 \big\|_{L^{p_1}(\mu)} \big\| f_2 \big\|_{L^{p_2}(\mu)}.
$$
This shows Theorem $\ref{L^p}$. 
\qed

The following Whitney decomposition originated in \cite{T-3} is the foundation of Lemma $\ref{good-lambda}$.
\begin{lemma}\label{Whitney decomposition}
If $\Omega \subset \Rn$ is open, $\Omega \neq \Rn$, then $\Omega$ can be decomposed as
$
\Omega = \bigcup_{i \in I} Q_i
$
where $\{Q_i\}_{i \in I}$ are closed dyadic cubes with disjoint interiors such that for some
constants $\rho > 20$ and $\rho_0 \geq 1$ the following holds:
\begin{enumerate}
\item [(1)] $10 Q_i \subset \Omega$ for each $i \in I$;
\item [(2)] $\rho Q_i \cap \Omega^c \neq \emptyset$ for each $i \in I$;
\item [(3)] For each cube $Q_i$, there are at most $\rho_0$ cubes $Q_j$ such that $10 Q_i \cap 10 Q_j \neq \emptyset$.
Further, for such cubes $Q_i$, $Q_j$, we have $\ell(Q_i) \simeq \ell(Q_j)$.
\end{enumerate}
Moreover, if $\mu$ is a positive Radon measure on $\Rn$ and $\mu(\Omega)<\infty$, there is a family of cubes $\{\widetilde{Q}_j\}_{j \in S}$, with $S \subset I$,
so that $Q_j \subset \widetilde{Q}_j \subset 1.1 Q_j$ , satisfying the following:
\begin{enumerate}
\item [(a)] Each cube $\widetilde{Q}_j$, $j \in S$, is $(9, 2 \rho_0)$-doubling and has $\mathfrak{C}$-small boundary.
\item [(b)] The collection $\{\widetilde{Q}_j \}_{j \in S}$ is pairwise disjoint.
\item [(c)] it holds that
\begin{equation}\label{j-S}
\mu \Big( \bigcup_{j \in S} \widetilde{Q}_j \Big) \geq \frac{1}{8 \rho_0} \mu(\Omega).
\end{equation}
\end{enumerate}
\end{lemma}
\noindent\textbf{Proof of Lemma $\ref{good-lambda}$.}
Applying Lemma $\ref{Whitney decomposition}$, one can get a family of dyadic cubes $\{Q_i\}_{i \in I}$ with disjoint interior such that $\Omega_{\xi} = \bigcup_{i \in I}Q_i$ and $\rho
Q_i \cap \Omega_{\xi}^c \neq \emptyset$. The collection $\{\widetilde{Q}_j \}_{j \in S}$ satisfies all properties of lemma. From the assumption in Theorem $\ref{L^p}$ and the fact
that the cubes $\{\widetilde{Q}_j\}_{j \in S}$ have $\mathfrak{C}$-small boundary and are $(9, 2 \rho_0)$-doubling, it follows that there exists subset $G_j \subset \widetilde{Q}_j$
with $\mu(G_j) \geq \theta \mu(\widetilde{Q}_j)$ such that $g_{\lambda}^* : \mathfrak{M}(\Rn) \times \mathfrak{M}(\Rn) \rightarrow L^{\frac12,\infty}(\mu \lfloor G_j)$, with norm
bounded uniformly on $j \in S$.
By the inequality $(\ref{j-S})$, we have
\begin{align*}
\mathscr{F}
&:=\mu \big(\big\{x;g_{\lambda,\mu,t_0}^*(f_1,f_2)(x) > (1+\epsilon)\xi, M_{\mu}f_1(x) M_{\mu}f_2(x) \leq \delta \xi  \big\}\big) \\
&\leq \mu \Big(\Omega_{\xi} \setminus \bigcup_{j \in S} \widetilde{Q}_j \Big) +  \sum_{j \in S} \mu(\widetilde{Q}_j \setminus G_j)
+ \sum_{j \in S} \mu (E_j) \\
&\leq \Big(1-\frac{\theta}{8 \rho_0}\Big) \mu(\Omega_\xi)  +
\sum_{j \in S} \mu (E_j),
\end{align*}
where $E_j:=\big\{x \in G_j;g_{\lambda,\mu,t_0}^*(f_1,f_2)(x) > (1+\epsilon)\xi, M_{\mu}f_1(x) M_{\mu}f_2(x) \leq \delta \xi  \big\}$.
To bound $\mathscr{F}$, we will prove that
\begin{equation}\label{subset}
E_j \subset \big\{x \in \widetilde{Q}_j; g_{\lambda,\mu,t_0}^*(f_1 \mathbf{1}_{2 \widetilde{Q}_j}, f_2 \mathbf{1}_{2 \widetilde{Q}_j})(x) > \epsilon \xi/2 \big\}.
\end{equation}
Once $(\ref{subset})$ is obtained, we by weak type bound deduce that
\begin{align*}
\mu(E_j)
&\leq \mu \big(\big\{x \in G_j;  g_{\lambda,\mu,t_0}^*(f_1 \mathbf{1}_{2 \widetilde{Q}_j}, f_2 \mathbf{1}_{2 \widetilde{Q}_j})(x) > \epsilon \xi/2 \big\}\big) 
\\
&\leq \frac{c}{(\epsilon \xi)^{1/2}} \prod_{i=1}^2 \bigg(\int_{2 \widetilde{Q}_j} |f_i| d\mu\bigg)^{1/2}.
\end{align*}
We may assume that there exists $x_0 \in \widetilde{Q}_j$ such that $M_{\mu}f_1(x_0) M_{\mu}f_2(x_0)\leq \delta \xi$, then
\begin{align*}
\mu(E_j)
&\leq \frac{c}{(\epsilon \xi)^{1/2}} \prod_{i=1}^2 \bigg(\int_{Q(x_0,4\ell(\widetilde{Q}_j))} |f| d\mu \bigg)^{1/2} \\
&\leq \frac{c}{(\epsilon \xi)^{1/2}} \mu\big(Q(x_0,4\ell(\widetilde{Q}_j))\big) M_{\mu} f_1(x_0)^{1/2} M_{\mu} f_2(x_0)^{1/2}  \\
&\leq c \delta^{1/2} \epsilon^{-1/2} \mu(10 Q_j)
\leq 2c \rho_0 \delta^{1/2} \epsilon^{-1/2} \mu(Q_j).
\end{align*}
Hence, we have
$$
\mathscr{F}
\leq \Big(1-\frac{\theta}{8 \rho_0}\Big) \mu(\Omega_\xi) + \widetilde{c} \delta^{1/2} \epsilon^{-1/2} \sum_{j \in S} \mu(Q_j)
\leq \Big(1-\frac{\theta}{16 \rho_0}\Big) \mu(\Omega_\xi),
$$
if we choose $\delta=\delta(\epsilon)$ small enough.

We now turn to demonstrate $(\ref{subset})$. Set $x \in \widetilde{Q}_j$ satisfying $g_{\lambda,\mu,t_0}^*(f_1,f_2)(x) > (1+\epsilon)\xi$ and
$M_{\mu}f_1(x) M_{\mu}f_2(x) \leq \delta \xi$. It is enough to show
\begin{equation}\label{2-Qi}
g_{\lambda,\mu,t_0}^*(f_1 \mathbf{1}_{2 \widetilde{Q}_j}, f_2 \mathbf{1}_{2 \widetilde{Q}_j})(x) > \epsilon \xi/2.
\end{equation}
By sub-linear property, it is enough to control
\begin{equation}\label{Rn-2Qj}
g_{\lambda,\mu,t_0}^*(f_1 \mathbf{1}_{\Rn \setminus 2 \widetilde{Q}_j}, f_2)(x)
+ g_{\lambda,\mu,t_0}^*(f_1 \mathbf{1}_{2 \widetilde{Q}_j}, f_2 \mathbf{1}_{\Rn \setminus 2 \widetilde{Q}_j})(x)
\leq (1+\epsilon/2) \xi.
\end{equation}

To analyze the contribution of $g_{\lambda,\mu,t_0}^*(f_1 \mathbf{1}_{\Rn \setminus 2 \widetilde{Q}_j}, f_2)(x)$, take $x' \in \rho \widetilde{Q}_j \setminus \Omega_{\xi}$.
We may assume that $t_0 < 2 \rho \ell(\widetilde{Q}_j)$. Then $g_{\lambda,\mu,t_0}^*(f_1,f_2)(x') \leq \xi$ and
\begin{equation}\label{N-N-N}
g_{\lambda,\mu,t_0}^*(f \mathbf{1}_{\Rn \setminus 2 \widetilde{Q}_j}, f_2)(x)
\leq \mathcal{J}_1 + \mathcal{J}_2 + \mathcal{J}_3 + \mathcal{J}_4,
\end{equation}
where
\begin{align*}
\mathcal{J}_1
&:= \bigg(\int_{0}^{2 \rho \ell(\widetilde{Q}_j)} \int_{\Rn} \Big(\frac{t}{t + |x - y|}\Big)^{m \lambda}
|\Theta_t^\mu (f_1 \mathbf{1}_{\Rn \setminus 2 \widetilde{Q}_j}, f_2)(y)|^2
\frac{d\mu(y) dt}{t^{m+1}}\bigg)^{1/2}, \\
\mathcal{J}_2
&:= \bigg(\int_{2 \rho \ell(\widetilde{Q}_j)}^{\infty} \int_{\Rn} \Big(\frac{t}{t + |x' - y|}\Big)^{m \lambda}
|\Theta_t^\mu (f_1,f_2)(y)|^2
\frac{d\mu(y) dt}{t^{m+1}}\bigg)^{1/2}, \\
\mathcal{J}_3
&:= \bigg(\int_{2 \rho \ell(\widetilde{Q}_j)}^{\infty} \int_{\Rn} \Big(\frac{t}{t + |x' - y|}\Big)^{m \lambda}
|\Theta_t^\mu (f_1 \mathbf{1}_{2 \widetilde{Q}_j}, f_2)(y)|^2
\frac{d\mu(y) dt}{t^{m+1}}\bigg)^{1/2}, \\
\mathcal{J}_4
&:= \bigg| \bigg(\int_{2 \rho \ell(\widetilde{Q}_j)}^{\infty} \int_{\Rn} \Big(\frac{t}{t + |x' - y|}\Big)^{m \lambda}
|\Theta_t^\mu (f_1 \mathbf{1}_{\Rn \setminus 2 \widetilde{Q}_j}, f_2)(y)|^2
\frac{d\mu(y) dt}{t^{m+1}}\bigg)^{1/2} \\
&{} \qquad-  \bigg(\int_{2 \rho \ell(\widetilde{Q}_j)}^{\infty} \int_{\Rn} \Big(\frac{t}{t + |x - y|}\Big)^{m \lambda}
|\Theta_t^\mu (f_1 \mathbf{1}_{\Rn \setminus 2 \widetilde{Q}_j}, f_2)(y)|^2 \frac{d\mu(y) dt}{t^{m+1}}\bigg)^{1/2} \bigg|.
\end{align*}
From Lemma $\ref{U(f)}$ and the inequality
\begin{equation}\label{Out}
\int_{\Rn \setminus 2 \widetilde{Q}_j}\frac{t^{\alpha/4} |f_2(z_2)|}{(t+|x-z_2|)^{m+\alpha/4}} d\mu(z_2)
\lesssim \min\big\{1, t^{\alpha/4} \ell(\widetilde{Q}_j)^{-\alpha/4}\big\} M_{\mu}(f_2)(x),
\end{equation}
it follows that
\begin{align*}
\mathcal{J}_1
&\lesssim \bigg(\int_{0}^{2 \rho \ell(\widetilde{Q}_j)} \mathscr{U}_t(f_1 \mathbf{1}_{\Rn \setminus 2 \widetilde{Q}_j}, f_2)(x)^2 \frac{dt}{t} \bigg)^{1/2} \\
&\lesssim \bigg(\int_{0}^{2 \rho \ell(\widetilde{Q}_j)} \frac{t^{\alpha/4}}{\ell(\widetilde{Q}_j)^{\alpha/4}} \frac{dt}{t} \bigg)^{1/2}  M_{\mu}(f_1)(x) M_{\mu}(f_2)(x) \\
&\lesssim M_{\mu}(f_1)(x) M_{\mu}(f_2)(x)
\leq \delta \xi.
\end{align*}
Since $t_0 < 2 \rho \ell(\widetilde{Q}_j)$,
$$
\mathcal{J}_2 \leq g_{\lambda,\mu,t_0}^*(f_1,f_2)(x') \leq \xi.
$$
Moreover, Lemma $\ref{U(f)}$ and the following inequality
\begin{equation}\label{In}
\int_{2 \widetilde{Q}_j}\frac{t^{\alpha/4} |f_1(z_1)|}{(t+|x-z_1|)^{m+\alpha/4}} d\mu(z_1)
\lesssim \min\big\{1, t^{-m} \mu(2 \widetilde{Q}_j)\big\} M_{\mu}(f_1)(x),
\end{equation}
indicate that
\begin{align*}
\mathcal{J}_3
&\lesssim \bigg(\int_{2 \rho \ell(\widetilde{Q}_j)}^{\infty} \mathscr{U}_t(f_1 \mathbf{1}_{2 \widetilde{Q}_j}, f_2)(x)^2 \frac{dt}{t} \bigg)^{1/2} \\
&\lesssim \bigg(\int_{2 \rho \ell(\widetilde{Q}_j)}^{\infty} \frac{\mu(2 \widetilde{Q}_j)}{t^{2m}} \frac{dt}{t} \bigg)^{1/2}  M_{\mu}(f_1)(x) M_{\mu}(f_2)(x) \\
&\lesssim M_{\mu}(f_1)(x) M_{\mu}(f_2)(x)
\leq \delta \xi.
\end{align*}
From the sub-linearity and Lemma $\ref{T(f)}$, it follows that
$$
\mathcal{J}_4
\leq \mathscr{T}(f_1^{\infty},f_2^0)(x) + \mathscr{T}(f_1^{\infty},f_2^{\infty})(x)
\lesssim M_{\mu}(f_1)(x) M_{\mu}(f_2)(x)
\leq \delta \xi.
$$

Next, we consider the contribution of $g_{\lambda,\mu,t_0}^*(f_1 \mathbf{1}_{2 \widetilde{Q}_j}, f_2 \mathbf{1}_{\Rn \setminus 2 \widetilde{Q}_j})(x)$.
Combining Lemma $\ref{U(f)}$ with the estimates $(\ref{Out})$ and $(\ref{In})$, we deduce that
\begin{align*}
&g_{\lambda,\mu,t_0}^*(f_1 \mathbf{1}_{2 \widetilde{Q}_j}, f_2 \mathbf{1}_{\Rn \setminus 2 \widetilde{Q}_j})(x)  \\
&\quad \leq \bigg(\int_{0}^{\infty} \mathscr{U}_t(f_1 \mathbf{1}_{2 \widetilde{Q}_j}, f_2 \mathbf{1}_{\Rn \setminus 2 \widetilde{Q}_j})(x)^2 \frac{dt}{t} \bigg)^{1/2} \\
&\quad\lesssim \bigg(\int_{0}^{\infty} \min\bigg\{\frac{\mu(2 \widetilde{Q}_j)}{t^{2m}}, \frac{t^{\alpha/4}}{\ell(\widetilde{Q}_j)^{\alpha/4}} \bigg\} \frac{dt}{t} \bigg)^{1/2} \\
&\quad\leq \bigg(\int_{0}^{\ell(\widetilde{Q}_j)} \frac{t^{\alpha/4}}{\ell(\widetilde{Q}_j)^{\alpha/4}} \frac{dt}{t}
+ \int_{\ell(\widetilde{Q}_j)}^{\infty} \frac{\mu(2 \widetilde{Q}_j)}{t^{2m}} \frac{dt}{t} \bigg)^{1/2} M_{\mu}(f_1)(x) M_{\mu}(f_2)(x)\\
&\quad \lesssim M_{\mu}(f_1)(x) M_{\mu}(f_2)(x)
\leq \delta \xi.
\end{align*}
Consequently, the above estimates indicate the inequality $(\ref{Rn-2Qj})$ holds for small enough $\delta=\delta(\epsilon)$. 
\qed


\section{Big Piece Bilinear Local $T1$ Theorem}\label{Sec-big}
In this section, we will prove Theorem $\ref{Big piece}$. In the proof, a probabilistic reduction and the martingale decomposition are essential. The fundamental tools we need are random dyadic grid and good cube, which can be found in \cite{Ht, NTV-02, NTV-03}.
\subsection{Random dyadic grids and good/bad cubes}
Let $\D_0$ be the standard dyadic grids on $\Rn$.
That is,
$$
\D_0 := \bigcup_{k \in \Z}\D_0^{k}, \ \ \D_0^k := \big\{2^{k}([0, 1)^n + m); k \in \Z, \ m \in \Z^n \big\}.
$$
For a binary sequence $w = \{ w_j \}_{j \in \Z} \in \Omega := (\{0,1\}^n)^\Z$, we define we define the new dyadic grid
$$
\D_{w}^k := \Big\{I + w := I + \sum_{j:j<k} 2^{j} w_j; I \in \D_0^k \Big\}.
$$
Then we will get the general dyadic systems of the form
$$
\D_w := \bigcup_{k \in \Z} \D^k_{w}.$$
There is a natural product probability structure on $\Omega$.

A cube $I \in \D$ is said to be good if there exists a $J \in \D$ with $\ell(J) \geq 2^r \ell(I)$ such that $\dist(I,\partial J) > \ell(I)^{\gamma} \ell(J)^{1-\gamma}$. Otherwise, $I$
is called bad.
Here $r \in \Z_+$ is a fixed large enough parameter, and $\gamma =\frac{\alpha}{2(m+\alpha)}$.

\subsection{Martingale difference operators}
Let us introduce the martingale difference operator as follows :
$$
\Delta_Q f = \sum_{Q' \in \ch(Q)} \big(\langle f \rangle_{Q'} - \langle f \rangle_{Q} \big) \mathbf{1}_{Q'}.
$$
We define the average operators :
$$
E_Q f = \langle f \rangle_Q \mathbf{1}_Q \quad\text{and}\quad 
E_{2^k}f = \sum_{Q \in \mathcal{D}, \ell(Q) = 2^k} E_Q f.
$$
Then there holds for any $s \in \Z$
\begin{eqnarray}
\label{f} f &=& \sum_{\substack{I \in \mathcal{D} \\ \ell(Q) \leq 2^s}} \Delta_Q f + \sum_{\substack{I \in \mathcal{D} \\ \ell(Q) = 2^s}} E_{Q} f, \ \
\text{in } L^2(\mu) \text{ and } \mu-a.e.. \\
\label{E-f} E_{2^k}f &=& \sum_{\substack{I \in \mathcal{D} \\ 2^k < \ell(Q) \leq 2^s}} \Delta_Q f + \sum_{\substack{I \in \mathcal{D} \\ \ell(Q) = 2^s}} E_{Q} f.
\end{eqnarray}

After preliminaries, we turn to showing Theorem $\ref{Big piece}$.

First of all, we prove the existence of $G_Q$ in Theorem $\ref{Big piece}$.
Set $G_Q := Q \setminus (H_Q \cup S_Q)$, 
\begin{equation}\label{SQ}
S_{Q} :=\big\{x \in Q; g_{\lambda,\mu,Q}^*(1_Q,1_Q)(x) > \zeta_0 \big\} 
\quad \text{and}\quad \zeta_0^{p_0}=\frac{2C_0}{1-\delta_0}.
\end{equation}
Using the weak type assumption $(\ref{Weak})$, we have
$$
\mu(G_Q)
\geq \mu(Q) - \mu(H_Q) - \mu(S_Q \setminus H_Q)
\geq \bigg(1 - \delta_0 - \frac{C_0}{\zeta_0^{p_0}}\bigg) \mu(Q)
=\frac{1-\delta_0}{2} \mu(Q).
$$
\subsection{Back to the global testing condition}
By size condition, it yields that
$$
|\theta_{t}^{\mu}(\mathbf{1}_Q,\mathbf{1}_Q)(y)| \lesssim \frac{\mu(Q)^2}{t^{2m}},
$$
which indicates that
\begin{align*}
G_{\infty}(x)
&:=\bigg(\int_{\ell(Q)}^{\infty}\int_{\Rn} \Big(\frac{t}{t+|x-y|}\Big)^{m\lambda} |\theta_t^\mu (\mathbf{1}_Q,\mathbf{1}_Q)(y)|^2 \frac{d\mu(y)}{t^m}\frac{dt}{t}\bigg)^{1/2} \\
&\leq C_1 \mu(Q)^2 \bigg(\int_{\ell(Q)}^{\infty} \frac{1}{t^{4m}} \frac{dt}{t}\bigg)^{1/2}
\leq C_2 \frac{\mu(Q)^2}{\ell(Q)^{2m}} \leq C_3.
\end{align*}
Accordingly, we get
\begin{align*}
&\sup_{\zeta > 0} \zeta^{p_0} \mu \big(\{x \in Q \setminus H_Q; g_{\lambda,\mu}^*(\mathbf{1}_Q,\mathbf{1}_Q)(x) > \zeta \}\big) \\
&\leq \sup_{\zeta > 0} \zeta^{p_0} \mu \big(\{x \in Q \setminus H_Q; g_{\lambda,\mu,Q}^*(\mathbf{1}_Q,\mathbf{1}_Q)(x) > \zeta/2 \}\big) \\
&\quad + \sup_{\zeta > 0} \zeta^{p_0} \mu \big(\{x \in Q \setminus H_Q; G_{\infty}(x) > \zeta/2 \}\big) \\
&\leq 2^{p_0}(C_0 + C_3^{p_0}) \mu(Q)
:= \widetilde{C}_0 \mu(Q).
\end{align*}
This is equivalent to
\begin{equation}\label{mu-Q}
\sup_{\zeta > 0} \zeta^{p_0} \mu \lfloor Q \big(\{x \in \Rn \setminus H_Q; g_{\lambda,\mu \lfloor Q}^*(\mathbf{1},\mathbf{1})(x) > \zeta \}\big)
\leq \widetilde{C}_0 \mu \lfloor Q(\Rn).
\end{equation}
Moreover, the desired result is
\begin{equation}\label{1-G}
\big\| \mathbf{1}_G g_{\lambda,\mu \lfloor Q}^*(\vec{f}) \big\|_{L^p(\mu \lfloor Q)}
\lesssim \prod_{i=1}^2 ||f_i||_{L^{p_i}(\mu \lfloor Q)},\ \text{ for each } f_i \in L^{p_i}(\mu).
\end{equation}
Therefore, we are reduced to demonstrating $(\ref{mu-Q})$ implies $(\ref{1-G})$ for $\mu$ replacing $\mu \lfloor Q$.


\subsection{Reductions}
In this subsection, our goal is to reduce the proof of $(\ref{1-G})$.
\subsubsection{\textbf{Discarding bad cubes.}}
We may assume that $\big\| \mathbf{1}_{G} g_{\lambda,\mu}^* \big\|_{L^p(\mu)} < \infty$, which can be got applying the similar argument in Proposition $3.1$ \cite{CX-2}. For
convenience, we denote
$$
\psi(x,t):= \bigg(\int_{\Rn} \Big(\frac{t}{t+|x-y|}\Big)^{m\lambda} |\Theta_t^{\mu} (\vec{f})(y)|^2 \frac{d\mu(y)}{t^m}\bigg)^{1/2}.
$$
Then we have
\begin{align*}
\big\| \mathbf{1}_G g_{\lambda,\mu}^*(\vec{f}) \big\|_{L^p(\mu)}
&=\bigg\| \mathbf{1}_G \bigg(\sum_{R \in \mathcal{D}_w} \mathbf{1}_{R} \int_{\ell(R)/2}^{\ell(R)} \psi(x,t)
\frac{dt}{t}\bigg)^{1/2} \bigg\|_{L^p(\mu)} \\
&\leq \mathbb{E}_w \bigg\| \mathbf{1}_G \bigg(\sum_{\substack{R \in \mathcal{D}_w \\ R: good}} \mathbf{1}_{R} \int_{\ell(R)/2}^{\ell(R)} \psi(x,t)
\frac{dt}{t}\bigg)^{1/2} \bigg\|_{L^p(\mu)} \\
&\quad +  \mathbb{E}_w \bigg\| \mathbf{1}_G \bigg(\sum_{\substack{R \in \mathcal{D}_w \\ R : bad}} \mathbf{1}_{R} \int_{\ell(R)/2}^{\ell(R)} \psi(x,t)
\frac{dt}{t}\bigg)^{1/2} \bigg\|_{L^p(\mu)} \\
&:=\Sigma_{good} + \Sigma_{bad}.
\end{align*}
Now we show the following :
\begin{equation}\label{Bad-1}
\Sigma_{bad} \leq 1/2 \big\| \mathbf{1}_G g_{\lambda,\mu}^*(\vec{f}) \big\|_{L^p(\mu)}.
\end{equation}
It follows from H\"{o}lder inequality that
\begin{align*}
\Sigma_{bad}
\leq \bigg\{\int_{\Rn} \mathbf{1}_G(x) \mathbb{E}_w \bigg(\sum_{\substack{R \in \mathcal{D}_w \\ R : bad}} \mathbf{1}_{R}(x)
\int_{\ell(R)/2}^{\ell(R)} \psi(x,t) \frac{dt}{t}\bigg)^{p/2} \bigg\}^{1/p}.
\end{align*}
Thus, it suffices to prove
\begin{equation}\label{Bad-2}
\mathbb{E}_w \bigg(\sum_{\substack{R \in \mathcal{D}_w \\ R : bad}} \mathbf{1}_{R}(x) \int_{\ell(R)/2}^{\ell(R)} \psi(x,t) \frac{dt}{t}\bigg)^{p/2}
\leq \frac12 \bigg(\int_{0}^{\infty} |\psi(x,t)|^2 \frac{dt}{t} \bigg)^{p/2}.
\end{equation}
Note that
$$
\mathbb{E}_w(\mathbf{1}_{bad}(R+w)) \leq c(r) \rightarrow 0 \text{ as } r \rightarrow \infty.
$$
The result can be found in \cite{NTV-03}.
If $p \leq 2$, Jensen's inequality implies that
\begin{equation}\aligned\label{p<2}
&\mathbb{E}_w \bigg(\sum_{\substack{R \in \mathcal{D}_w \\ R : bad}} \mathbf{1}_{R}(x) \int_{\ell(R)/2}^{\ell(R)}
|\psi(x,t)|^2 \frac{dt}{t}\bigg)^{p/2} \\
&\leq \bigg(\mathbb{E}_w \sum_{\substack{R \in \mathcal{D}_w \\ R : bad}} \mathbf{1}_{R}(x) \int_{\ell(R)/2}^{\ell(R)}
|\psi(x,t)|^2 \frac{dt}{t}\bigg)^{p/2} \\
&= \bigg(\sum_{R \in \mathcal{D}_0} \mathbb{E}_w(\mathbf{1}_{bad}(R+w)) \mathbb{E}_w \bigg(\mathbf{1}_{R+w}(x) \int_{\ell(R)/2}^{\ell(R)}
|\psi(x,t)|^2 \frac{dt}{t}\bigg)\bigg)^{p/2} \\
&\leq c(r)^{p/2} \bigg(\int_{0}^{\infty} |\psi(x,t)|^2 \frac{dt}{t} \bigg)^{p/2}.
\endaligned
\end{equation}
If $p > 2$, we have
\begin{align*}
&\mathbb{E}_w \bigg(\sum_{\substack{R \in \mathcal{D}_w \\ R : bad}} \mathbf{1}_{R}(x) \int_{\ell(R)/2}^{\ell(R)}
|\psi(x,t)|^2 \frac{dt}{t}\bigg)^{p/2} \\
&=\mathbb{E}_w \bigg(\sum_{\substack{R \in \mathcal{D}_w \\ R : bad}} \mathbf{1}_{R}(x) \int_{\ell(R)/2}^{\ell(R)}
|\psi(x,t)|^2 \frac{dt}{t}\bigg)
\bigg(\sum_{\substack{R \in \mathcal{D}_w \\ R : bad}} \mathbf{1}_{R}(x) \int_{\ell(R)/2}^{\ell(R)}
|\psi(x,t)|^2 \frac{dt}{t}\bigg)^{p/2-1} \\
&\leq \mathbb{E}_w \bigg(\sum_{\substack{R \in \mathcal{D}_w \\ R : bad}} \mathbf{1}_{R}(x) \int_{\ell(R)/2}^{\ell(R)}
|\psi(x,t)|^2 \frac{dt}{t}\bigg) \bigg(\int_{0}^{\infty} |\psi(x,t)|^2 \frac{dt}{t} \bigg)^{p/2-1}\\
&\leq c(r) \bigg(\int_{0}^{\infty} |\psi(x,t)|^2 \frac{dt}{t} \bigg)^{p/2},
\end{align*}
where we used the conclusion $(\ref{p<2})$ for $p=2$.
Therefore, by taking large enough $r$, we obtain $(\ref{Bad-2})$ and $(\ref{Bad-1})$, which gives that
$$
\big\| \mathbf{1}_G g_{\lambda,\mu}^*(\vec{f}) \big\|_{L^p(\mu)}
\leq 2 \Sigma_{good}.
$$

With the monotone convergence theorem, it is enough to deduce that there exists a constant $C > 0$ such that for any $s \in \N$ and $w \in \Omega$, we have
$$
\bigg\| \mathbf{1}_G \bigg(\sum_{\substack{R \in \mathcal{D}_w, \ell(R) \leq 2^s \\ R: good}}
\mathbf{1}_{R} \int_{\ell(R)/2}^{\ell(R)} |\psi(x,t)|^2 \frac{dt}{t}\bigg)^{1/2} \bigg\|_{L^p(\mu)}
\leq C \prod_{i=1}^2 ||f_i||_{L^{p_i}(\mu)}.
$$
From now on, $w$ is fixed, simply denote $\mathcal{D}_{good}=\{R; R \in \D_w, R \text{ is good}\}$,
$s_{t,G}(y,z_1,z_2)=s_{t,G}(y,z_1,z_2) \mathbf{1}_{\Rn \setminus G}(x)$, and
$\Theta_{t,G}^{\mu} (\vec{f})(y)=\int_{\R^{2n}} s_{t,G}(y,z_1,z_2) f_1(z_1) f_2(z_2) d\mu(z_1) d\mu(z_2)$.
It is easy to check that $s_{t,G}$ satisfies the Size condition and H\"{o}lder conditions.
We are to reduced to showing that
\begin{align*}
\bigg\|\bigg(\sum_{\substack{R \in \mathcal{D}_{good}\\ \ell(R) \leq 2^s}}
\mathbf{1}_{R} \int_{\ell(R)/2}^{\ell(R)} \int_\Rn \vartheta_t(\cdot,y)
\Theta^{\mu}_{t,G}(f_1, f_2)(y) \Big|^2 \frac{d\mu dt}{t^{m+1}}\bigg)^{\frac12} \bigg\|_{L^p(\mu)}
\lesssim \prod_{i=1}^2 ||f_i||_{L^{p_i}(\mu)}.
\end{align*}
\subsubsection{\textbf{Martingale difference decomposition.}}
The proof in this section is motivated by the ideas in \cite{MV}.

For convenience, when $\ell(Q)=2^s$, $\Delta_Q$ is understood as $\Delta_Q  + E_Q$.
Using the martingale difference decomposition $(\ref{f})$ and $(\ref{E-f})$, we have
\begin{align*}
&\Theta^{\mu}_{t,G}(f_1,f_2)
=\sum_{\substack{Q_1 \in \mathcal{D} \\ \ell(Q_1) \leq 2^s}} \Theta^{\mu}_{t,G}
\bigg(\Delta_{Q_1} f_1, \sum_{\substack{Q_2 \in \mathcal{D} \\ \ell(Q_1) \leq \ell(Q_2) \leq 2^s}} \Delta_{Q_1} f_2 \bigg) \\
&\qquad\qquad\qquad\quad + \sum_{\substack{Q_2 \in \mathcal{D} \\ \ell(Q_2) \leq 2^s}} \Theta^{\mu}_{t,G}
\bigg(\sum_{\substack{Q_1 \in \mathcal{D} \\ \ell(Q_2) < \ell(Q_1) \leq 2^s}} \Delta_{Q_1} f_1, \Delta_{Q_2} f_2 \bigg) \\
&=\sum_{\substack{Q_1 \in \mathcal{D} \\ \ell(Q_1) \leq 2^s}} \Theta^{\mu}_{t,G}\big(\Delta_{Q_1} f_1, E_{2^{-1}\ell(Q_1)} f_2\big)
 + \sum_{\substack{Q_2 \in \mathcal{D} \\ \ell(Q_2) < 2^s}} \Theta^{\mu}_{t,G}\big(E_{\ell(Q_2)} f_1, \Delta_{Q_2} f_2\big).
\end{align*}
Since the second one is much simpler, we focus on estimating the following term:
{\small
\begin{align*}
\G:=\bigg\|\bigg(\sum_{\substack{R \in \mathcal{D}_{good}\\ \ell(R) \leq 2^s}}
\mathbf{1}_{R} \int_{\ell(R)/2}^{\ell(R)} \int_\Rn \vartheta_t
\Big| \sum_{\substack{Q_1 \in \mathcal{D} \\ \ell(Q_1) \leq 2^s}} \Theta^{\mu}_{t,G}(\Delta_{Q_1} f_1, E_{2^{-1}\ell(Q_1)}f_2) \Big|^2 \frac{d\mu dt}{t^{m+1}}\bigg)^{\frac12} \bigg\|_{L^p(\mu)}.
\end{align*}
}

\subsection{Main estimates.}
In this subsection, we shall bound $\G$. For fixed cube $R \in \D_{good}$, we split the cubes $Q_1 \in \D$ into four cases:
\begin{enumerate}
\item  $\Xi_1 :=\big\{Q_1; \ell(Q_1) < \ell(R)\big\}$;
\vspace{0.2cm}
\item  $\Xi_2 := \big\{Q_1; \ell(Q_1) \geq \ell(R)$, $d(Q_1,R) > \ell(R)^{\gamma} \ell(Q_1)^{1-\gamma}\big\}$;
\vspace{0.2cm}
\item  $\Xi_3 := \big\{Q_1; \ell(R) \leq \ell(Q_1) \leq 2^r \ell(R)$, $d(Q_1,R) \leq \ell(R)^{\gamma}\ell(Q_1)^{1-\gamma}\big\}$;
\vspace{0.2cm}
\item  $\Xi_4 :=\big\{Q_1; \ell(Q_1) > 2^r \ell(R)$, $d(Q_1,R) \leq \ell(R)^{\gamma} \ell(Q_1)^{1-\gamma}\big\}$.
\end{enumerate}
Hence, we obtain correspondingly four terms, $\G_1$, $\G_2$, $\G_3$ and $\G_4$.

The following two lemmas will be used at certain key points below. The first one was shown in \cite{MV}.
\begin{lemma}\label{delta-QR}
Denote
$$
\delta(Q,R)=\frac{\ell(Q)^{\alpha/2} \ell(R)^{\alpha/2}}{D(Q,R)^{m+\alpha}},
$$
where $D(Q,R)=\ell(Q) + \ell(R) + d(Q,R)$ and $\alpha > 0$. Then for every
$x_Q \geq 0$, there holds that
$$
\bigg\| \bigg(\sum_{R \in \D} \mathbf{1}_R \Big(\sum_{Q \in \D} \delta(Q,R) \mu(Q) x_Q \Big)^2\bigg)^{1/2} \bigg\|_{L^p(\mu)}
\lesssim \bigg\| \Big(\sum_{Q \in \D} x_Q^2 \mathbf{1}_Q \Big)^{1/2}\bigg\|_{L^p(\mu)} .
$$
\end{lemma}
\begin{lemma}\label{A-Q1}
Let $0 < \alpha \leq m(\lambda - 2)/2$. Let $Q_1$ and $R$ be given cubes and $(x,t) \in W_R$.
If $Q_1 \in \Xi_1 \cup \Xi_2 \cup \Xi_3$, then there holds that
\begin{align*}
\mathcal{A}_{Q_1}(x,t)&:=
\bigg( \int_{\Rn} \vartheta_t(x,y) \big|\Theta^{\mu}_{t,G}(\Delta_{Q_1} f_1, E_{2^{-1}\ell(Q_1)}f_2)(y) \big|^2 \frac{d\mu(y)}{t^m}\bigg)^{1/2} \\
&\lesssim M_{m}(M_{\D}f_2)(x) \cdot \delta(Q_1,R) \big\|\Delta_{Q_1} f_1 \big\|_{L^1(\mu)} .
\end{align*}
\end{lemma}
\begin{proof}
$(1)$ The condition $Q_1 \in \Xi_1$ implies the vanishing property $\int \Delta_{Q_1} f_1 d\mu =0$. Then by H\"{o}lder condition, we have
{\small
\begin{align*}
&\big|\Theta_{t,G}^{\mu}(\Delta_{Q_1} f_1, E_{2^{-1}\ell(Q_1)}f_2)(y)\big| \\
&= \bigg|\int_{\Rn} \int_{Q_1} \big(s_t(y,z_1,z_2) - s_t(y,c_{Q_1},z_2)\big) \Delta_{Q_1} f_1(z_1) E_{2^{-1}\ell(Q_1)}f_2(z_2) d\mu\bigg| \\
&\lesssim \int_{Q_1} \frac{\ell(Q_1)^\alpha}{(t+|y-z_1|)^{m+\alpha}}|\Delta_{Q_1} f_1(z_1)| d\mu(z_1)
\int_{\Rn} \frac{t^\alpha M_{\D} f_2(z_2)}{(t+|y-z_2|)^{m+\alpha}} d\mu(z_2) \\
&\lesssim M_{m}(M_{\D} f_2)(x) \cdot \int_{Q_1} \frac{\ell(Q_1)^\alpha}{(t+|y-z_1|)^{m+\alpha}}|\Delta_{Q_1} f_1(z_1)| d\mu(z_1).
\end{align*}
}
Thus, it follows from Minkowski's inequality that
\begin{align*}
\mathcal{A}_{Q_1}(x,t)
&\lesssim  M_{m}(M_{\D} f_2)(x) \int_{Q_1} |\Delta_{Q_1} f_1(z_1)|  
\\
&\quad\times \bigg(\int_{\Rn} \Big(\frac{t}{t+|x-y|}\Big)^{m \lambda} \frac{\ell(Q_1)^{2\alpha}}{(t+|y-z_1|)^{2(m+\alpha)}}
\frac{d\mu(y)}{t^m}\bigg)^{1/2} d\mu(z_1) .
\end{align*}

In order to treat the contribution of the inner integral, we split the domain $\Rn=\big\{y;|y-z_1| > d(Q_1,R)/2\big\} \cup \big\{y;|y-z_1| \leq d(Q_1,R)/2\big\} =: E_1 \cup E_2 $.
If $|y-z_1| > d(Q_1,R)/2$, there holds that $t + |y-z_1| \gtrsim \ell(R) + d(Q_1,R) \simeq  D(Q_1,R)$.
Thus, it follows that
\begin{equation}\label{E1}
\bigg(\int_{E_1} \Big(\frac{t}{t+|x-y|}\Big)^{m \lambda} \frac{\ell(Q_1)^{2\alpha}}{(t+|y-z_1|)^{2(m+\alpha)}} \frac{d\mu(y)}{t^m}\bigg)^{1/2}
\lesssim \delta(Q_1,R).
\end{equation}
If $y:|y-z_1| \leq d(Q_1,R)/2$, then $|x-y| \geq |x-z_1| - |y-z_1| \geq d(Q_1,R)/2$ and
$$
\Big(\frac{t}{t+|x-y|}\Big)^{m \lambda}
\lesssim \frac{t^{2(m+\alpha)}}{(\ell(R)+d(Q_1,R))^{2(m+\alpha)}}
\simeq \frac{t^{2(m+\alpha)}}{D(Q_1,R)^{2(m+\alpha)}}.$$
Therefore, we obtain that
\begin{equation}\label{E2}
\bigg(\int_{E_2} \Big(\frac{t}{t+|x-y|}\Big)^{m \lambda} \frac{\ell(Q_1)^{2\alpha}}{(t+|y-z_1|)^{2(m+\alpha)}} \frac{d\mu(y)}{t^m}\bigg)^{1/2}
\lesssim \delta(Q_1,R),
\end{equation}
We have used the inequality $\int_{\Rn} \big(\frac{t}{t+|y-z_1|}\big)^{\tau} \frac{d\mu(y)}{t^m} \lesssim 1$ for any $\tau > m$ in $(\ref{E1})$ and $(\ref{E2})$.
Collection the above estimates, we deduce the desired result. \\
$(2)$ The condition $Q_1 \in \Xi_2 \cup \Xi_3$ indicates that
\begin{equation}\label{R-R-Q1-R}
\frac{\ell(R)^\alpha}{(\ell(R)+d(Q_1,R))^{m+\alpha}}
\lesssim \delta(Q_1,R).
\end{equation}
Actually, if $Q_1 \in \Xi_3$, it is easy to see that $\ell(Q_1) \simeq \ell(R) \simeq D(Q_1,R)$, which gives $(\ref{R-R-Q1-R})$. It remains to consider the case $Q_1 \in \Xi_2$.
If  $\ell(Q_1) \leq d(Q_1,R)$, it is obvious that
\begin{align*}
\frac{\ell(R)^\alpha}{(\ell(R)+d(Q_1,R))^{m+\alpha}}
\lesssim \frac{\ell(R)^\alpha}{D(Q_1,R)^{m+\alpha}}
\leq \delta(Q_1,R).
\end{align*}
If $\ell(Q_1) > d(Q_1,R)$, then $\ell(Q_1) \simeq D(Q_1,R)$. Together with $d(Q_1,R) > \ell(R)^{\gamma} \ell(Q_1)^{1-\gamma}$ and
$\gamma = \frac{\alpha}{2(m+\alpha)}$, this gives that
\begin{align*}
\ell(Q_1) = \bigg(\frac{\ell(Q_1)}{\ell(R)}\bigg)^\gamma \ell(R)^\gamma \ell(Q_1)^{1-\gamma}
< \bigg(\frac{\ell(Q_1)}{\ell(R)}\bigg)^\gamma d(Q_1,R),
\end{align*}
and
\begin{align*}
\frac{\ell(R)^\alpha}{(\ell(R)+d(Q_1,R))^{m+\alpha}}
\leq \frac{\ell(R)^\alpha}{d(Q_1,R)^{m+\alpha}}
\leq \frac{\ell(Q_1)^{\alpha/2} \ell(R)^{\alpha/2}}{\ell(Q_1)^{m+\alpha}}
\simeq \delta(Q_1,R).
\end{align*}

The size condition implies that
\begin{align*}
\big|\Theta_{t,G}^{\mu}(\Delta_{Q_1} f_1, E_{2^{-1}\ell(Q_1)}f_2)(y)\big|
\lesssim M_{m}(M_{\D} f_2)(x) \int_{Q_1} \frac{t^\alpha |\Delta_{Q_1} f_1(z_1)|}{(t+|y-z_1|)^{m+\alpha}} d\mu(z_1).
\end{align*}
The rest of arguments are similar to those in the above case. This completes the proof.
\end{proof}

\subsubsection{\textbf{Parts $\G_1$, $\G_2$ and $\G_3$.}}\label{sec-less}
Based on the above lemmas, we deal with the three terms $\G_1$, $\G_2$ and $\G_3$ uniformly. Applying Minkowski's inequality and Lemma $\ref{delta-QR}$, we have
\begin{align*}
\G_1
&\leq \bigg\|\bigg\{\sum_{\substack{R \in \mathcal{D}_{good}\\ \ell(R) \leq 2^s}}
\mathbf{1}_{R} \bigg[\sum_{\substack{Q_1 \in \mathcal{D} \\ \ell(Q_1) < \ell(R)}} \bigg(\int_{\ell(R)/2}^{\ell(R)} \mathcal{A}_{Q_1}(x,t)^2
\frac{dt}{t}\bigg)^{1/2}\bigg]^2\bigg\}^{1/2} \bigg\|_{L^p(\mu)} \\
&\lesssim \bigg\| M_{m}(M_{\D}f_2) \bigg\{\sum_{\substack{R \in \mathcal{D}_{good}\\ \ell(R) \leq 2^s}} \mathbf{1}_{R}
\bigg[\sum_{\substack{Q_1 \in \mathcal{D} \\ \ell(Q_1) \leq 2^s}} \delta(Q_1,R) \big\|\Delta_{Q_1} f_1 \big\|_{L^1(\mu)} \bigg]^2\bigg\}^{1/2} \bigg\|_{L^p(\mu)}.
\end{align*}
Furthermore, H\"{o}lder's inequality and $L^p(\mu)$ boundedness of the maximal operators give that
\begin{align*}
\G_1
&\leq \big\| M_{m}(M_{\D}f_2)\big\|_{L^{p_2}} \bigg\| \bigg\{\sum_{\substack{R \in \mathcal{D}_{good}\\ \ell(R) \leq 2^s}}
\mathbf{1}_{R} \bigg[\sum_{\substack{Q_1 \in \mathcal{D} \\ \ell(Q_1) \leq 2^s}} \delta(Q_1,R) \big\|\Delta_{Q_1} f_1 \big\|_{L^1} \bigg]^2\bigg\}^{\frac12} \bigg\|_{L^{p_1}} \\
&\lesssim \big\| f_2 \big\|_{L^{p_2}(\mu)} \bigg\| \bigg(\sum_{\substack{Q_1 \in \mathcal{D} \\ \ell(Q_1) \leq 2^s}} \langle |\Delta_{Q_1} f_1| \rangle_{Q_1}^2
\mathbf{1}_{Q_1}\bigg)^{1/2} \bigg\|_{L^{p_1}(\mu)}.
\end{align*}
Note that
$$
\bigg\| \bigg(\sum_{\substack{Q_1 \in \mathcal{D} \\ \ell(Q_1) \leq 2^s}} \langle |\Delta_{Q_1} f_1| \rangle_{Q_1}^2 \mathbf{1}_{Q_1}\bigg)^{\frac12} \bigg\|_{L^{p_1}}
\lesssim \bigg\| \bigg(\sum_{\substack{Q_1 \in \mathcal{D} \\ \ell(Q_1) \leq 2^s}} |\Delta_{Q_1} f_1|^2 \bigg)^{\frac12} \bigg\|_{L^{p_1}}
\lesssim \big\| f_1 \big\|_{L^{p_1}}.
$$
This shows $\G_1 \lesssim \big\| f_1 \big\|_{L^{p_1}(\mu)} \big\| f_2 \big\|_{L^{p_2}(\mu)}$.
The arguments for $\G_2$ and $\G_3$ are the same.
\qed
\subsubsection{\textbf{Part $\G_4$.}}
Let $R^{(k)} \in \mathcal{D}$ be the unique cube for which $R \subset R^{(k)}$ and $\ell(R^{(k)}) = 2^k \ell(R)$. In this case, it holds $R \subset Q_1$, since $R$ is good. Then we
write
\begin{align*}
\G_4 &= \bigg\|\bigg(\sum_{\substack{R \in \mathcal{D}_{good}\\ \ell(R) \leq 2^{s-r-1}}}
\mathbf{1}_{R} \int_{\ell(R)/2}^{\ell(R)} \int_\Rn \vartheta_t(\cdot,y) 
\\
&\qquad\times \Big| \sum_{k=r+1}^{s-\log_2{\ell(R)}} \Theta^{\mu}_{t,G}(\Delta_{R^{(k)}} f_1, E_{2^{-1}\ell(R^{(k)})}f_2)(y) \Big|^2 \frac{d\mu(y) dt}{t^{m+1}}\bigg)^{1/2} \bigg\|_{L^p(\mu)}.
\end{align*}
Note that
\begin{align}
\label{Delta-R} &\Delta_{R^{(k)}}f_1 = \mathbf{1}_{(R^{(k-1)})^c} \Delta_{R^{(k)}}f_1
  - \langle \Delta_{R^{(k)}}f_1 \rangle_{R^{(k-1)}} \mathbf{1}_{(R^{(k-1)})^c}
  + \langle \Delta_{R^{(k)}}f_1 \rangle_{R^{(k-1)}}, 
\\
\label{E-R} &E_{2^{-1} \ell(R^{(k)})}f_2
= \mathbf{1}_{(R^{(k-1)})^c} E_{\ell(R^{(k-1)})}f_2
  - \langle f_2 \rangle_{R^{(k-1)}} \mathbf{1}_{(R^{(k-1)})^c}
  + \langle f_2 \rangle_{R^{(k-1)}}.
\end{align}
Using $(\ref{Delta-R})$, we control $\G_4$ by three terms, in which the first two terms are denoted by $\G_{41}$ and $\G_{42}$. As for the term corresponding to
$\Theta^{\mu}_{t,G}(\langle \Delta_{R^{(k)}}f_1 \rangle_{R^{(k-1)}}, E_{2^{-1}\ell(R^{(k)})}f_2)$, it by $(\ref{E-R})$ is dominated by
other three parts denoted by $\G_{43}$, $\G_{44}$ and $\G_{par}$. If we set
\begin{align*}
\mathcal{N}_{k,1}(x,t)
&:=\bigg(\int_\Rn \vartheta_t \big|\Theta^{\mu}_{t,G}(\mathbf{1}_{(R^{(k-1)})^c} \Delta_{R^{(k)}}f_1, E_{2^{-1}\ell(R^{(k)})}f_2)(y) \big|^2 \frac{d\mu}{t^m}\bigg)^{\frac12}, \\
\mathcal{N}_{k,2}(x,t)
&:=\big|\langle \Delta_{R^{(k)}}f_1 \rangle_{R^{(k-1)}}\big| \bigg(\int_\Rn \vartheta_t \big|\Theta^{\mu}_{t,G}(\mathbf{1}_{(R^{(k-1)})^c}, E_{2^{-1}\ell(R^{(k)})}f_2)(y)
\big|^2 \frac{d\mu}{t^m}\bigg)^{\frac12}, \\
\mathcal{N}_{k,3}(x,t)
&:=\big|\langle \Delta_{R^{(k)}}f_1 \rangle_{R^{(k-1)}}\big| \bigg(\int_\Rn \vartheta_t \big|\Theta^{\mu}_{t,G}(1, \mathbf{1}_{(R^{(k-1)})^c} E_{\ell(R^{(k-1)})}f_2)(y) \big|^2
\frac{d\mu}{t^m}\bigg)^{\frac12}, \\
\mathcal{N}_{k,4}(x,t)
&:=\big|\langle \Delta_{R^{(k)}}f_1 \rangle_{R^{(k-1)}}\big|  \big|\langle f_2 \rangle_{R^{(k-1)}}\big|
\bigg(\int_\Rn \vartheta_t \big|\Theta^{\mu}_{t,G}(1, \mathbf{1}_{(R^{(k-1)})^c})(y) \big|^2 \frac{d\mu}{t^m}\bigg)^{\frac12},
\end{align*}
then
\begin{align*}
\G_{4j} \leq \bigg\|\bigg(\sum_{\substack{R \in \mathcal{D}_{good}\\ \ell(R) \leq 2^{s-r-1}}}
\mathbf{1}_{R} \bigg\{\sum_{k=r+1}^{s-\log_2{\ell(R)}} \bigg(\int_{\ell(R)/2}^{\ell(R)} \mathcal{N}_{k,j}(\cdot,t)^2 \frac{dt}{t} \bigg)^{\frac12} \bigg\}^2 \bigg)^{\frac12}
\bigg\|_{L^p(\mu)}.
\end{align*}
Let us dominate $\mathcal{N}_{k,j}(x,t)$. From Lemma $\ref{U(f)}$ and goodness of $R$, it follows that
\begin{align*}
&\mathcal{N}_{k,1}(x,t)
\lesssim \int_{(R^{(k-1)})^c} \frac{t^{\alpha/4}}{(t+|x-z_1|)^{m+\alpha/4}} \big|\Delta_{R^{(k)}}f_1(z_1)\big| d\mu(z_1) \\
&\qquad\qquad \times \int_{\Rn} \frac{t^{\alpha/4}}{(t+|x-z_2|)^{m+\alpha/4}} \big|E_{2^{-1}\ell(R^{(k)})}f_2(z_2)\big| d\mu(z_2) \\
&\lesssim  \ell(R)^{\alpha/4} d(R, \partial R^{(k-1)})^{m+\alpha/4} \big\| \Delta_{R^{(k)}}f_1 \big\|_{L^1(\mu)} \cdot M_m(M_{\D} f_2)(x) \\
&\lesssim  \ell(R)^{\alpha/4} \ell(R)^{\gamma(m+\alpha/4)} \ell(R^{(k-1)})^{(1-\gamma)(m+\alpha/4)} \big\| \Delta_{R^{(k)}}f_1 \big\|_{L^1(\mu)} M_m(M_{\D} f_2)(x) \\
&\lesssim  2^{-\alpha k/2} \big\langle |\Delta_{R^{(k)}} f_1| \big\rangle_{R^{(k)}}  M_m(M_{\D} f_2)(x) .
\end{align*}
Applying Lemma $\ref{U(f)}$ again, we have
\begin{align*}
\mathcal{N}_{k,2}(x,t)
&\lesssim \big|\langle \Delta_{R^{(k)}}f_1 \rangle_{R^{(k-1)}}\big|  M_m(M_{\D} f_2)(x) \int_{(R^{(k-1)})^c} \frac{t^{\alpha/4} d\mu(z_1)}{(t+|x-z_1|)^{m+\alpha/4}} \\
&\lesssim \big|\langle\Delta_{R^{(k)}}f_1 \rangle_{R^{(k-1)}}\big|  M_m(M_{\D} f_2)(x)\int_{\Rn \setminus B(x,d(R,\partial R^{(k-1)}))} \frac{\ell(R)^{\alpha/4} d\mu(z_1)}{|x -z_1|^{m+\alpha/4}}\\
&\lesssim \ell(R)^{\alpha/4} d(R, \partial R^{(k-1)})^{-\alpha/4} \big|\langle\Delta_{R^{(k)}}f_1 \rangle_{R^{(k-1)}}\big| M_m(M_{\D} f_2)(x) \\
&\lesssim 2^{-\alpha k/8} \big\langle |\Delta_{R^{(k)}} f_1| \big\rangle_{R^{(k-1)}} M_m(M_{\D} f_2)(x).
\end{align*}
Similarly, it yields that
\begin{align*}
\mathcal{N}_{k,3}(x,t)
&\lesssim \big|\langle \Delta_{R^{(k)}}f_1 \rangle_{R^{(k-1)}}\big| \int_{(R^{(k-1)})^c} 
\frac{t^{\alpha/4} M_{\D} f_2(z_2) }{(t+|x-z_2|)^{m+\alpha/4}} d\mu(z_2)  \\
&\lesssim 2^{-\alpha k/8} \big\langle |\Delta_{R^{(k)}} f_1| \big\rangle_{R^{(k-1)}} M_m(M_{\D} f_2)(x),
\end{align*}
and
$$
\mathcal{N}_{k,4}(x,t)
\lesssim 2^{-\alpha k/8} \big\langle |\Delta_{R^{(k)}} f_1| \big\rangle_{R^{(k-1)}} M_m(M_{\D} f_2)(x).
$$
Consequently, by H\"{o}lder's inequality and Minkowski's inequality, we conclude that
{\small
\begin{align*}
\G_{41} &\leq \big\| M_m(M_{\D} f_2) \big\|_{L^{p_2}}
\bigg\|\bigg(\sum_{\substack{R \in \mathcal{D}_{good}\\ \ell(R) \leq 2^{s-r-1}}}
\mathbf{1}_{R} \bigg\{\sum_{k=r+1}^{s-\log_2{\ell(R)}} 2^{-\frac{\alpha}{8} k} \big\langle |\Delta_{R^{(k)}} f_1| \big\rangle_{R^{(k)}} \bigg\}^2 \bigg)^{\frac12} \bigg\|_{L^{p_1}} \\
&\lesssim \big\| f_2 \big\|_{L^{p_2}(\mu)}
\sum_{k=r+1}^{s-\log_2{\ell(R)}} 2^{-\alpha k/8} \bigg\|\bigg(\sum_{\substack{R \in \mathcal{D}_{good}\\ \ell(R) \leq 2^{s-r-1}}}
\mathbf{1}_{R} \big\langle |\Delta_{R^{(k)}} f_1| \big\rangle_{R^{(k)}}^2 \bigg)^{\frac12} \bigg\|_{L^{p_1}(\mu)} \\
&\lesssim \big\| f_2 \big\|_{L^{p_2}(\mu)} \bigg\|\bigg(\sum_{\substack{R \in \mathcal{D}\\ \ell(R) \leq 2^{s}}}
\mathbf{1}_{R} \big\langle |\Delta_{R} f_1| \big\rangle_{R}^2 \bigg)^{1/2} \bigg\|_{L^{p_1}(\mu)} \\
&\lesssim \big\| f_2 \big\|_{L^{p_2}(\mu)} \bigg\|\bigg(\sum_{\substack{R \in \mathcal{D}\\ \ell(R) \leq 2^{s}}}
|\Delta_{R} f_1|^2 \bigg)^{1/2} \bigg\|_{L^{p_1}(\mu)}
\lesssim \big\| f_1 \big\|_{L^{p_1}(\mu)} \big\| f_2 \big\|_{L^{p_2}(\mu)}.
\end{align*}
}
The other three parts can be controlled as follows. For $j=2,3,4$, there holds that
\begin{align*}
\G_{4j}
&\lesssim \big\| f_2 \big\|_{L^{p_2}(\mu)} \bigg\|\bigg(\sum_{\substack{R \in \mathcal{D} \\ \ell(R) \leq 2^{s-1}}}
\mathbf{1}_{R} \big\langle |\Delta_{R^{(1)}} f_1| \big\rangle_{R}^2 \bigg)^{\frac12} \bigg\|_{L^{p_1}(\mu)} \\
&= \big\| f_2 \big\|_{L^{p_2}(\mu)} \bigg\|\bigg(\sum_{\substack{R \in \mathcal{D} \\ \ell(R) \leq 2^{s}}} |\Delta_{R} f_1| \bigg)^{1/2} \bigg\|_{L^{p_1}(\mu)}
\lesssim \big\| f_1 \big\|_{L^{p_1}(\mu)} \big\| f_2 \big\|_{L^{p_2}(\mu)}.
\end{align*}
The remainder of this subsection is devoted to bounding the term $\G_{par}$.

\noindent\textbf{$\bullet$ Paraproduct estimate.}
Recall that
\begin{align*}
\G_{par}:=\bigg\|\bigg(\sum_{\substack{R \in \mathcal{D}_{good}\\ \ell(R) \leq 2^{s-r-1}}}
\mathbf{1}_{R} \int_{\ell(R)/2}^{\ell(R)} \int_\Rn \vartheta_t(\cdot,y)
\Big|  \Theta^{\mu}_{t,G}(1, 1)(y) \Big|^2
\frac{d\mu dt}{t^{m+1}}\bigg)^{1/2} \bigg\|_{L^p(\mu)}.
\end{align*}
where 
\[
\mathscr{A}_R := 
\sum_{k=r+1}^{s-\log_2{\ell(R)}} \langle \Delta_{R^{(k)}}f_1 \rangle_{R^{(k-1)}} \langle f_2 \rangle_{R^{(k-1)}} . 
\]
Splitting $\langle f_2 \rangle_{R^{(k-1)}}=\langle f_2 \rangle_{R^{(k)}} + \langle \Delta_{R^{(k)}} f_2 \rangle_{R^{(k-1)}}$, 
we dominate $\G_{par}$ by the corresponding two pieces denoted by $\G_{par}'$ and $\G_{par}''$.

To discuss the term $\G_{par}'$, write $\psi := \sum_{\substack{Q \in \mathcal{D} \\ \ell(Q) \leq 2^s}} \Delta_Q f_1 \cdot \langle f_2 \rangle_Q$. Observe that 
$$
\sum_{k=r+1}^{s-\log_2{\ell(R)}} \langle \Delta_{R^{(k)}}f_1 \rangle_{R^{(k-1)}} \langle f_2 \rangle_{R^{(k)}}
=\sum_{k=r+1}^{s-\log_2{\ell(R)}}  \langle \Delta_{R^{(k)}} \psi \rangle_{R^{(k-1)}}
= \langle \psi \rangle_{R^{(r)}}.
$$
Thereupon, it yields that
\begin{align*}
\G_{par}^{'}=\bigg\|\bigg(\sum_{\substack{Q \in \mathcal{D} \\ \ell(Q) \leq 2^{s-1}}}
\big|\langle \psi \rangle_Q\big|^2 a_Q^2\bigg)^{1/2} \bigg\|_{L^p(\mu)},
\end{align*}
if we denote
\begin{equation}\label{aQ}
a_Q(x) := \bigg(\sum_{\substack{R \in \mathcal{D}_{good}\\ R^{(r)}=Q }}
\mathbf{1}_{R} \big|\langle \psi \rangle_{R^{(r)}}\big|^2 \int_{\ell(R)/2}^{\ell(R)} \int_\Rn \vartheta_t 
\big| \Theta^{\mu}_{t,G}(1, 1)(y) \big|^2 \frac{d\mu dt}{t^{m+1}}\bigg)^{1/2}.
\end{equation}
Thus, Lemma $\ref{phi-aQ}$ implies that
\begin{align*}
\G_{par}^{'}
&\lesssim \Big\| \sum_{\substack{Q \in \mathcal{D} \\ \ell(Q) \leq 2^s}} \Delta_Q f_1 \cdot \langle f_2 \rangle_Q \Big\|_{L^p(\mu)}
\lesssim \bigg\| \bigg(\sum_{\substack{Q \in \mathcal{D} \\ \ell(Q) \leq 2^s}} |\Delta_Q f_1|^2  |\langle f_2 \rangle_Q| \bigg)^{1/2}\bigg\|_{L^p(\mu)} \\
&\lesssim \bigg\| M_{\D}f_2 \bigg(\sum_{\substack{Q \in \mathcal{D} \\ \ell(Q) \leq 2^s}} |\Delta_Q f_1|^2 \bigg)^{1/2}\bigg\|_{L^p(\mu)}
\\
&\lesssim \big\|M_{\D}f_2\big\|_{L^{p_2}(\mu)} \bigg\| \bigg(\sum_{\substack{Q \in \mathcal{D} \\ \ell(Q) \leq 2^s}} |\Delta_Q f_1|^2 \bigg)^{1/2}\bigg\|_{L^{p_1}(\mu)} 
\lesssim \big\| f_1 \big\|_{L^{p_1}(\mu)} \big\| f_2 \big\|_{L^{p_2}(\mu)}.
\end{align*}

In order to analyze $\G_{par}^{''}$, set $S(f):=\Big(\sum_{\substack{Q \in \mathcal{D} \\ \ell(Q) \leq 2^{s}}} |\Delta_Q f|^2 \Big)^{1/2}$. Then we get
\begin{align*}
&\bigg|\sum_{k=r+1}^{s-\log_2{\ell(R)}} \langle \Delta_{R^{(k)}}f_1 \rangle_{R^{(k-1)}} \langle \Delta_{R^{(k)}} f_2 \rangle_{R^{(k-1)}}\bigg| \\
&\leq \bigg(\sum_{k=r+1}^{s-\log_2{\ell(R)}} \big\langle |\Delta_{R^{(k)}}f_1| \big\rangle_{R^{(k-1)}}^2 \bigg)^{1/2}
     \bigg(\sum_{k=r+1}^{s-\log_2{\ell(R)}} \big\langle |\Delta_{R^{(k)}}f_2| \big\rangle_{R^{(k-1)}}^2 \bigg)^{1/2} \\
&\leq \big\langle S(f_1) S(f_2) \big\rangle_{R^{(r)}}^2.
\end{align*}
Together with Lemma $\ref{phi-aQ}$, this implies that
\begin{align*}
\G_{par}^{''}
&\lesssim \bigg\| \sum_{\substack{Q \in \mathcal{D} \\ \ell(Q) \leq 2^s}} \big\langle S(f_1) S(f_2) \big\rangle_{Q}^2 a_Q^2 \bigg\|_{L^p(\mu)}
\lesssim \big\| S(f_1) S(f_2) \big\|_{L^p(\mu)} \\
&\leq \big\| S(f_1) \big\|_{L^{p_1}(\mu)} \big\| S(f_2) \big\|_{L^{p_2}(\mu)}
\lesssim \big\| f_1 \big\|_{L^{p_1}(\mu)} \big\| f_2 \big\|_{L^{p_2}(\mu)}.
\end{align*}
So far, we have shown Theorem $\ref{Big piece}$.

\begin{lemma}\label{phi-aQ}
Let $1<q<\infty$ and $\{ a_Q \}_{Q \in \D}$ be the same as $(\ref{aQ})$. Then there holds that
$$
\mathscr{Z}:=
\bigg\| \bigg(\sum_{Q \in \D: \ell(Q) \leq 2^s} |\langle \phi \rangle_Q|^2 a_Q^2 \bigg)^{1/2} \bigg\|_{L^q(\mu)}
\lesssim ||\phi||_{L^q(\mu)}.
$$
\end{lemma}
\begin{proof}
We here follow the scheme of the proof in \cite{LPR}. Let us first introduce the principal cubes.
Let $\mathscr{F}_0$ be the set of maximal cubes $Q \in \D$ with $\ell(Q) \leq 2^s$. And inductively,
$$
\mathscr{F}_{k+1}:=\bigcup_{Q \in \mathscr{F}_k} \big\{Q' \subset Q;\ \langle |\phi| \rangle_{Q'} > 2 \langle |\phi| \rangle_{Q},\ Q' \in \D \text{ is maximal}
\big\}.
$$
Set $\mathscr{F} := \bigcup_{k=0}^{\infty} \mathscr{F}_k$. For any cube $Q \in \D$ with $\ell(Q) \leq 2^s$, we denote by $Q^a$ the minimal cube in $\mathscr{F}$ that contains $Q$.

It follows from the definition that $\langle |\phi| \rangle_{Q} \leq 2 \langle |\phi| \rangle_{Q^a}$. Moreover, by $(\ref{SQ})$, we have
$$
\sum_{Q \in \D: Q \subset F} a_Q(x)^2
\leq \mathbf{1}_{F}(x) g_{\lambda,\mu,G}^*(1,1)(x)^2
\lesssim \mathbf{1}_{F}(x),
$$
which implies that
\begin{align*}
\mathscr{Z} & = \bigg\| \bigg(\sum_{Q^a \in \mathscr{F}} \sum_{\substack{Q \in \D \\ Q \subset Q^a}}
 |\langle \phi \rangle_Q|^2 a_Q^2 \bigg)^{1/2} \bigg\|_{L^q(\mu)}
\\
&\lesssim \bigg\| \bigg(\sum_{F \in \mathscr{F}} \langle |\phi| \rangle_F^2 \mathbf{1}_F \bigg)^{1/2} \bigg\|_{L^q(\mu)}
\leq \bigg\| \sum_{F \in \mathscr{F}} \langle |\phi| \rangle_F \mathbf{1}_F \bigg\|_{L^q(\mu)}.
\end{align*}
By duality, there exists $g \in L^{q'}(\mu)$ with $||g||_{L^{q'}(\mu)}=1$, such that
\begin{align*}
\mathscr{Z}
&\lesssim \int_{\Rn} \sum_{F \in \mathscr{F}} \langle |\phi| \rangle_F \mathbf{1}_F(x) g(x) d\mu(x)
=\sum_{F \in \mathscr{F}} \langle |\phi| \rangle_F  \langle g \rangle_F \mu(F) \\
&\leq \Big(\sum_{F \in \mathscr{F}} \langle |\phi| \rangle_F^p  \mu(F) \Big)^{1/p}
      \Big(\sum_{F \in \mathscr{F}} \langle |g| \rangle_F^{p'}  \mu(F) \Big)^{1/{p'}} \\
&\lesssim ||\phi||_{L^q(\mu)} ||g||_{L^{q'}(\mu)}
=||\phi||_{L^q(\mu)},
\end{align*}
which is provided by Carleson embedding theorem. Hence, it only remains to show
$$
\sum_{F' \in \mathscr{F}: F' \subset F} \mu(F') \lesssim \mu(F),\ \text{for any } F \in \mathscr{F}.
$$
Write $E(F):= F \setminus \bigcup_{F' \in \ch_{\mathscr{F}}(F)} F'$. Then we have
$$
\mu(E(F)) \geq \frac12 \mu(F) \quad\text{and}\quad \{E(F)\}_{F \in \mathscr{F}} \text{ is a disjoint family}.
$$
Consequently, we deduce that
$$
\sum_{F' \in \mathscr{F}: F' \subset F} \mu(F')
\leq 2 \sum_{F' \in \mathscr{F}: F' \subset F} \mu(E(F'))
\leq 2 \mu(F).
$$
This completes the proof.
\end{proof}

\subsection*{Acknowledgements} The authors want to express their sincere thanks to the referee for his or her valuable remarks and
suggestions, which made this paper more readable. 



\begin{thebibliography}{00}

\bibitem{BH}T. A. Bui, M. Hormozi,
\emph{Weighted bounds for multilinear square functions},
Potential Anal. 46 (2017), 135--148.


\bibitem{CX-2}M. Cao, Q. Xue,
\emph{A non-homogeneous local $Tb$ theorem for Littlewood-Paley $g_{\lambda}^{*}$-function with $L^p$-testing condition}, Forum Math. 30 (2018), 457--478.

\bibitem{CX-3}M. Cao, Q. Xue, 
\emph{$L^p$ boundedness of non-homogeneous Littlewood-Paley $g^*_{\lambda, \mu}$-function with non-doubling measures}, https://arxiv.org/abs/1605.04649. 


\bibitem{CXY}X. Chen, Q. Xue, K. Yabuta,
\emph{On multilinear Littlewood-Paley operators},
Nonlinear Anal. 115 (2015), 25-40.

\bibitem{CDM}R. R. Coifman, D. Deng, Y. Meyer,
\emph{Domains de la racine carr\'{e}e de certains op\'{e}rateurs diff\'{e}rentiels accr\'{e}tifs},
Ann. Inst. Fourier (Grenoble) 33 (1983), 123--134.

\bibitem{CMM}R. R. Coifman, A. McIntosh, Y. Meyer,
\emph{L'integrale de Cauchy definit un operateur borne sur $L^2$ pour les courbes lips-chitziennes},
Ann. of Math. 116 (1982), 361--387.

\bibitem{CM}R. R. Coifman, Y. Meyer,
\emph{Au-del\`{a} des op\'{e}rateurs pseudo-diff\'{e}rentiels},
Asterisque 57 (1978).

\bibitem{DJK}B. Dahlberg, D. Jerison, C. Kenig,
\emph{Area integral estimates for elliptic differential operators with non-smooth coefficients},
Arkiv Mat. 22 (1984), 97--108.

\bibitem{DJ}G. David, J. L. Journe,
\emph{Une caract\'{e}risation des op\'{e}rateurs int\'{e}graux singuliers born\'{e}s sur $L^2(\Rn)$},
C. R. Math. Acad. Sci. Paris 296 (1983) 761--764.

\bibitem{FJK-1}E. B. Fabes, D. Jerison, C. Kenig,
\emph{Multilinear Littlewood-Paley estimates with applications to partial differential equations},
Proc. Natl. Acad. Sci. 79 (1982), 5746--5750.

\bibitem{FJK-2}E. B. Fabes, D. Jerison, C. Kenig,
\emph{Necessary and sufficient conditions for absolute continuity of elliptic harmonic measure},
Ann. of Math. 119 (1984), 121--141.

\bibitem{FJK-3}E. B. Fabes, D. Jerison, C. Kenig,
\emph{Multilinear square functions and partial differential equations},
Amer. J. Math. 107 (1985), 1325--1368.

\bibitem{F}C. Fefferman,
\emph{Inequalities for strongly singular convolution operators},
Acta Math. 124 (1970), 9--36.

\bibitem{FS}C. Feffrman, E. M. Stein,
\emph{$H^p$ spaces of several variables},
Acta Math. 129 (1972), 137--193.

\bibitem{GW}R. F. Gundy, R. L. Wheeden,
\emph{Weighted integral inequalities for the nontangential maximal function, Lusin area integral, and Walsh-Paley series},
Studia Math. 49 (1973), 101--118.

\bibitem{HXMY}S. He, Q. Xue, T. Mei, K. Yabuta,
\emph{Existence and boundedness of multilinear Littlewood-Paley operators on Campanato spaces},
J. Math. Anal. Appl. 432 (2015), 86--102.


\bibitem{Ht}T. Hyt\"{o}nen, 
\emph{The sharp weighted bound for general Calderon-Zygmund operators},  
Ann. Math., (2) 175 (2012), 1473--1506.


\bibitem{Ht2}T. Hyt\"{o}nen, 
\emph{A framework for non-homogeneous analysis on metric spaces, and the RBMO space of Tolsa}, Publ. Mat. 54 (2010), 485--504.

\bibitem{K}D. S. Kurtz,
\emph{Littlewood-Paley operators on $BMO$},
Proc. Amer. Math. Soc. 99 (1987), 657--666.

\bibitem{K-1}D. S. Kurtz,
\emph{Rearrangement inequalities for Littlewood-Paley operators},
Math. Nachrichten 133 (1987), 71--90.


\bibitem{LPR}M. T. Lacey, S. Petermichl, M. C. Reguera,
\emph{Sharp $A_2$ inequality for Haar shift operators},
Math. Ann. 348 (2010), 127--141.

\bibitem{L}A. K. Lerner,
\emph{On pointwise estimates for the Littlewood-Paley operators},
Proc. Amer. Math. Soc.131 (2002), 1459--1469.

\bibitem{L-1}A. K. Lerner,
\emph{On some sharp weighted norm inequalities},
J. Funct. Anal. 232 (2006), 477--494.

\bibitem{L-2}A. K. Lerner,
\emph{On some weighted norm inequalities for Littlewood-Paley operators},
Illinois J. Math. 52 (2008), 653--666.

\bibitem{L-3}A. K. Lerner,
\emph{On sharp aperture-weighted estimates for square functions},
J. Fourier Anal. Appl. 20 (2014), 784--800.

\bibitem{LP-1}J. E. Littlewood and R. E. A. C. Paley,
\emph{Theorems on Fourier series and power series},
J. London Math. Soc. 6 (1931), 230--233.

\bibitem{LP-2}J. Littlewood, R. Paley,
\emph{Theorems on Fourier series and power series}, II,
Proc. Lond. Math. Soc. 42 (1936), 52--89.

\bibitem{MZ}J. Marcinkiewicz, A. Zygmund,
\emph{On a theorem of Lusin},
Duke Math. J. 4 (1938), 473--485.

\bibitem{MV}H. Martikainen, E. Vuorinen,
\emph{Dyadic-probabilistic methods in bilinear analysis},
https://arxiv.org/abs/1609.

\bibitem{MR}B. Muckenhoupt, R. L. Wheeden,
\emph{Norm inequalities for the Littlewood-Paley function $g_\lambda^*$},
Trans. Amer. Math. Soc. 191 (1974), 95--111.

\bibitem{NTV-02}F. Nazarov, S. Treil, A. Volberg,
\emph{Accretive system $Tb$-theorems on nonhomogeneous spaces},
Duke Math. J. 113 (2002), 259--312.

\bibitem{NTV-03}F. Nazarov, S. Treil, A. Volberg,
\emph{The $Tb$-theorem on non-homogeneous spaces},
Acta Math. 190 (2003), 151--239.

\bibitem{RS}L. de Rosa, C. Segovia,
\emph{One-sided Littlewood-Paley theory},
J. Fourier Anal. Appl. 3 (1997), 933--957.

\bibitem{SY}M. Sakamoto, K. Yabuta,
\emph{Boundedness of Marcinkiewicz functions},
Studia. Math. 135 (1999), 103--142.

\bibitem{SW}C. Segovia, R. L. Wheeden,
\emph{On the function gt and the heat equation},
Studia Math. 37 (1970), 57--93.

\bibitem{SXY}S. Shi, Q. Xue and K. Yabuta,
\emph{On the boundedness of multilinear Littlewood-Paley $g_{\lambda}^*$ function},
J. Math. Pures Appl. 101 (2014), 394--413.

\bibitem{S-58}E. M. Stein,
\emph{On the functions of Littlewood-Paley, Lusin, and Marcinkiewicz},
Trans. Amer. Math. Soc. 88 (1958), 430--466.

\bibitem{S-61}E. M. Stein,
\emph{On some function of Littlewood-Paley and Zygmund},
Bull. Amer. Math. Soc. 67 (1961), 99--101.

\bibitem{T-3}X. Tolsa,
\emph{Analytic capacity, the Cauchy transform, and non-homogeneous Calder\'{o}n-Zygmund theory},
Progress in Mathematics, Vol. 307, Birkh\"{a}user Verlag, Basel, 2014.

\bibitem{XY}Q. Xue, J. Yan,
\emph{On multilinear square function and its applications to multilinear
Littlewood-Paley operators with non-convolution type kernels},
J. Math. Anal. Appl. 422 (2015), 1342--1362.

\bibitem{Z}A. Zygmund,
\emph{On certain integrals},
Trans. Amer. Math. Soc. 55 (1944), 170--204.

\end{thebibliography}
\end{document}